\documentclass{amsart}
\usepackage{amssymb,graphicx,mathtools,amsthm,microtype,mathrsfs,crossreftools,color,amsmath,setspace,geometry}

\definecolor{darkblue}{rgb}{0.1,0,0.5}
\usepackage[linktocpage=true,colorlinks,linkcolor=darkblue,anchorcolor=green,citecolor=darkblue]{hyperref}

\usepackage{thmtools}
\usepackage[noabbrev,capitalise]{cleveref}

\theoremstyle{theorem}
\newtheorem{theorem}{Theorem}[section]
\newtheorem{lemma}[theorem]{Lemma}
\newtheorem{proposition}[theorem]{Proposition}
\newtheorem{corollary}[theorem]{Corollary}
\newtheorem{definition}[theorem]{Definition}

\theoremstyle{remark}
\newtheorem{example}[theorem]{Example}
\AtEndEnvironment{example}{\hfill$\blacktriangleleft$}
\newtheorem{remark}[theorem]{Remark}
\AtEndEnvironment{remark}{\hfill$\blacktriangleleft$}
\newtheorem{problem}[theorem]{Problem}
\newtheorem{conjecture}[theorem]{Conjecture}

\numberwithin{equation}{section}

\DeclareMathOperator{\Tel}{Tel}

\DeclareMathOperator{\Id}{Id}

\DeclareMathOperator{\precision}{prec} 

\DeclareMathOperator{\poly}{poly}
\DeclareMathOperator{\image}{Im}

\DeclareMathOperator{\Top}{top}

\renewcommand{\AA}{\mathbb{A}}

\newcommand{\DD}{\mathbb{D}}

\newcommand{\NN}{\mathbb{N}}

\newcommand{\PP}{\mathbb{P}}
\newcommand{\QQ}{\mathbb{Q}}
\newcommand{\RR}{\mathbb{R}}

\newcommand{\ZZ}{\mathbb{Z}}

\newcommand{\ccP}{\textsf{\upshape P}}
\newcommand{\ccNP}{\textsf{\upshape NP}}
\newcommand{\ccPpoly}{\textsf{\upshape P/poly}}

\newcommand{\ccB}{\textsf{\upshape B}}
\newcommand{\ccD}{\textsf{\upshape D}}
\newcommand{\ccG}{\textsf{\upshape G}}
\newcommand{\ccT}{\textsf{\upshape T}}

\newcommand{\ccH}{\textsf{\upshape H}}
\newcommand{\ccM}{\textsf{\upshape M}}
\newcommand{\ccN}{\textsf{\upshape N}}
\newcommand{\ccK}{\textsf{\upshape K}}
\newcommand{\ccV}{\textsf{\upshape V}}

\newcommand{\ccPPT}{\textsf{\upshape PPT}}

\newcommand{\ccTELIC}{\textsf{\upshape TELIC}}

\newcommand{\cD}{\mathcal{D}}

\newcommand{\cP}{\mathcal{P}}

\title[Correspondences in computational and dynamical complexity I]{Correspondences in computational and dynamical complexity I}
\author[S. Everett]{Samuel Everett}
\address{University of Chicago}
\email{same@uchicago.edu}
\date{}

\begin{document}


\begin{abstract}We begin development of a method for studying dynamical systems using concepts from computational complexity theory. We associate families of decision problems, called telic problems, to dynamical systems of a certain class. These decision problems formalize finite-time reachability questions for the dynamics with respect to natural coarse-grainings of state space. Our main result shows that complexity-theoretic lower bounds have dynamical consequences: if a system admits a telic problem for which every decider runs in time $2^{\Omega(n)}$, then it must have positive topological entropy. This result and others lead to methods for classifying dynamical systems through proving bounds on the runtime of algorithms solving their associated telic problems, or by constructing polynomial-time reductions between telic problems coming from distinct dynamical systems.\end{abstract}

\maketitle
\tableofcontents
\pagebreak

\setstretch{1.06}

\section{Introduction}

The aim of the three papers composing this work is to develop a method for studying dynamical systems using ideas from computational complexity theory, and conversely.
Our work centers around the study of a class of search and decision problems intrinsic to dynamical systems.
We show that there is a bidirectional relationship between the complexity of a dynamical system and the tractability of the problems associated with the system.
More explicitly, one may prove a dynamical system possesses certain properties by proving bounds on the runtime of algorithms solving the problems associated with that system.
In the other direction, the structure of a dynamical system limits the classes of algorithms solving its associated search and decision problems.
The present paper as well as the third \cite{everett3} in the collection focus on the former direction, while the second paper \cite{everett2} studies the latter direction; the three papers of this work can be read in any order.

\subsection{Overview of the main results of this paper}
The main results of this paper require certain terminology which we introduce informally now, postponing the formal definitions for the next sections.

Let the pair $(X, T)$ denote a dynamical system, where $X\subseteq \RR^m$ is the state space and $T:X\rightarrow X$ is a transformation. By a \emph{topological dynamical system} we mean $X\subset \RR^m$ is compact and $T$ is continuous.
We shall say $(X, T)$ is \emph{efficiently discretizable} (\cref{defDynDisc}) if there exists a polynomial-time Turing machine that on input $x \in X \cap \QQ^m$, and positive integers $k$ and $r$ encoded in unary, returns $w \in X\cap \QQ^m$ for which $\|w-T^k(x)\|_2 \leq 1/2^r$.
Furthermore, if the map is invertible and the inverse is efficiently discretizable, we say the dynamical system is \emph{efficiently discretizable with inverse.}
Many standard examples of dynamical systems are efficiently discretizable, such as rigid circle rotations, the logistic map, the H\'enon map, the Chirikov standard map, and a broad class of Lipschitz maps more generally (see \cref{secDiscretization} for detailed discussion).

Recall that a \emph{decision problem} in computational complexity theory is a computational task that asks for a yes-or-no answer about whether a given input satisfies a specified property. A Turing machine \emph{decides} or \emph{solves} a decision problem if it accepts positive instances and rejects negative instances.
We now give an informal description of \emph{telic problems},\footnote{The word \emph{telic} comes from the Greek word ``telos" for end/goal, or ``reaching a goal," which aptly characterizes the problems considered in this paper as a class of reachability problem.} the central decision problem considered in this paper (\cref{deftelic}).
Fix an efficiently discretizable dynamical system $(X, T)$.
Given an integer $n\geq 1$, a description of a closed ball $B \subset X$, and a description of a function $g:\{0,1\}^n\rightarrow X$ carrying the length $n$ binary strings into points in $X$ (\cref{defFeasibleQuery}), an instance of a \emph{telic problem coming from $(X, T)$} is the problem of deciding whether there is an $s \in \{0,1\}^n$ such that $T^{n}(g(s)) \in B$. Any efficiently discretizable dynamical system admits an associated class of telic problems, each specified by a family of functions $\{g^{(n)}\}_{n=1}^\infty$ and closed balls $\{B^{(n)}\}_{n=1}^\infty$. We say a telic problem is \emph{easy} or \emph{tractable} if there exists a polynomial-time Turing machine that decides it. We shall say the telic problem is \emph{hard} or \emph{intractable} otherwise.

There is a close connection between the behavior of a dynamical system and the tractability of its associated telic problems, as demonstrated in the following theorem.

\begin{theorem}[\cref{thm:entropyHardness}]\label{thm:entropyHardness1}
Let $(X, T)$ be an efficiently discretizable topological dynamical system. If $(X, T)$ admits a telic problem such that every decider for the language runs in time $2^{\Omega(n)}$, then the topological entropy of the system is positive.
\end{theorem}

As a consequence of \cref{thm:entropyHardness1}, breakthroughs in our ability to prove algorithmic lower-bounds give way to tools for studying dynamical systems.
Furthermore, if $(X, T)$ is a system admitting an intractable telic problem, and this problem is polynomial-time reducible to a telic problem coming from a different system $(Y, S)$, then $(Y, S)$ has positive topological entropy by \cref{thm:entropyHardness1}. The use of this corollary is that it is often easy to construct polynomial-time reductions between decision problems, when they exist.
In \cite{everett3} we apply these techniques in a number of examples.

We say a subset $A \subseteq X$ of state space \emph{contains} a telic problem if, morally speaking, given a family of target sets $\{B^{(n)}\}_{n=1}^\infty$ and functions $\{g^{(n)}\}_{n=1}^\infty$ (where $g^{(n)}:\{0,1\}^n\rightarrow X$) that compose a telic problem, it holds that $B^{(n)} \subseteq A$ and $\image(g^{(n)}) \subset A$ for every $n\geq 1$ (see \cref{defContainment}). We say a set is \emph{pure} if all telic problems contained in the set are decided by Turing machines halting in polynomial time. A system $(X, T)$ is \emph{transitive} if, for any two nonempty open sets $U, V \subseteq X$, there is an $N$ such that $T^N(U) \cap V \neq \emptyset$.

If a dynamical system is transitive and admits hard telic problems, one would expect the hard telic problems to be common: that is, they cannot be contained in a null set. Indeed, we obtain the following dichotomy.

\begin{theorem}[\cref{thmTopDichotomy}]\label{thmTopDichotomy1}
Let $X \subset \RR^d$ be compact and connected, and let $(X, T)$ be an efficiently discretizable and transitive topological dynamical system with inverse.
Then either $X$ is pure, or every pure set $A \subset X$ is contained in the complement of an open and dense subset of $X$.
\end{theorem}

The following theorem essentially establishes the complexity of solving telic problems as a topological invariant for certain systems.

\begin{theorem}[\cref{thmConjugacyHardness}]\label{thmConjugacyHardness1}
Let $X \subset \RR$ be a closed interval, and let $(X, T)$ and $(X, S)$ be two efficiently discretizable dynamical systems.
Suppose $\varphi:X\rightarrow X$ is a conjugacy of the two systems that is efficiently discretizable with inverse.
Then $(X, T)$ admits a hard telic problem if and only if $(X, S)$ admits a hard telic problem.
\end{theorem}

There is a classification theory of efficiently discretizable dynamical systems based on correspondences between their associated telic problems.
This can be realized by introducing the concept of \emph{reductions} between dynamical systems. If every telic problem coming from an efficiently discretizable dynamical system $(X, T)$ is polynomial-time reducible (by either Karp or Turing reductions) to some telic problem coming from an efficiently discretizable system $(Y, S)$, we say $(X, T)$ \emph{reduces} to $(Y, S)$ and write $(X, T) \leq_p (Y, S)$.
The relation $\leq_p$ is a preorder on the class of efficiently discretizable dynamical systems. This induces an equivalence relation $\sim_p$ so that $(X, T) \sim_p (Y, S)$ if and only if $(Y, S) \leq_p (X, T)$ and $(X, T) \leq_p (Y, S)$. If $(X, T) \sim_p (Y, S)$ we say the two systems are \emph{Turing equivalent}.

The computational tractability of telic problems is an invariant of Turing equivalence. A central question is whether there are also dynamical invariants.

\begin{problem}\label{probMain}
If $(X, T)$ and $(Y, S)$ are efficiently discretizable dynamical systems, and $(X, T) \sim_p (Y, S)$, does this imply $(X, T)$ and $(Y, S)$ share certain nontrivial dynamical properties?
\end{problem}

A positive answer would reveal an intimate connection between the two disciplines. The importance of this problem is demonstrated in \cref{propReductionEntropy}, asserting the existence of nontrivial invariants implies existence of hard telic problems (in particular, this would imply $\ccP \neq \ccNP$; the dynamicist seeking to apply methods from dynamics to complexity should begin with \cite{everett2}, from which there appear to be a plethora of tractable starting points).

This problem is also closely related to the existence of \emph{pseudorandom dynamical systems}. We say an efficiently discretizable dynamical system is \emph{pseudorandom} if it admits an intractable telic problem.
The pseudorandomness property essentially states the system has a ``stable" form of stochasticity. That is, upon discretizing the system, although the true stochasticity available by way of the continuum is lost, an artifact remains that 
drives the computational hardness.
The existence of pseudorandom dynamical systems connects with foundational open problems in computational complexity theory, particularly in the foundations of cryptography and the existence of one-way functions (see \cref{secProblems}).

\subsection{Historical context and motivation}\label{secBackground}

Our work adds to what has in effect been a decades-long program to bridge all areas of dynamical systems and the theory of computation.
However, perhaps the spirit of our papers is most related to deep work of Freedman, Nabutovsky, and Weinberger \cite{freedman1998topological,freedman2009complexity,weinberger2020computers,nabutovsky1996algorithmic}
who, very roughly, demonstrate the extent to which certain topological properties are entwined with computational properties.

Many of the most successful connections between dynamics and the theory of computation have been made in an impressive literature linking computability theory and dynamics.
In contrast, there has been less development between complexity theory and dynamical systems; although there has been more activity in this area recently (see below).

Indeed, historically emphasis has been placed on applying tools from computability theory and algorithmic randomness to dynamical systems, specifically computable analysis \cite{weihrauch2012computable} and Kolmogorov complexity \cite{li2008introduction}.
These lines of work have yielded a number of strong results, illuminating deep dynamical and physical structures \cite{crisanti1994applying,hiura2019microscopic,da1991undecidability,da1994undecidable,zurek1989algorithmic,braverman2015space, kreinovich2006towards,bennett1982thermodynamics}.
Salient in the literature is also work concerning the computability of certain objects arising in dynamical systems, such as strange attractors and specifically Julia sets \cite{gracca2018computing,braverman2009constructing,braverman2006non,braverman2009computability,rojas2019computational,bournez2017polynomial,gracca2024robust}.
This relates in part to the development of a separate line of inquiry treating dynamical systems as computers, appealing to the notion that movement and computation are two sides of the same coin.

The genesis of the later field is attributed to Moore \cite{moore1990unpredictability,moore1991generalized}, who communicated the idea of dynamical systems exhibiting a degree of unpredictability so strong basic state-prediction questions are undecidable.
This notion has since led to deep exploration of dynamical systems capable of computation \cite{cotler2024computational,koiran1999closed,PhysRevLett.83.1463,bournez2013computation} and related study of cellular automata and Turing machines treated as dynamical systems \cite{bruera2024topological,kuurka2001topological,kuurka1997languages,kuurka1997topological}.
Outstanding among such work is a sequence of papers inquiring into the ability of flows to exhibit universal computation, in part motivated by a creative idea of Tao for using universal computation to show existence of blow-up solutions to the Navier-Stokes equations \cite{cardona2021constructing,cardona2022turing,cardona2024hydrodynamic,tao2016finite,tao2017universality,cardona2024towards}.

In contrast, less emphasis has been placed on developing connections between computational complexity theory and dynamical systems, and results to this end are more wide-ranging.
Broadly, the majority of the literature connecting the two areas in some way may be roughly classified as follows.
\begin{enumerate}
\item Work studying the computational complexity of solving certain problems coming from a restricted class of dynamical system, such as state prediction problems, and problems related to computing the limiting behavior of a typical orbit \cite{applebaum2010cryptography,barrett2007predecessor,barrett2006complexity,papadimitriou2016computational,buss1990predictability,dudko2016poly,binder2006computational,dufort1997dynamics,rosenkrantz2024synchronous,ogihara2017computational,rojas2025algorithmic,de2018characterization,rojas2019computational,boyer2015mu}.
\item Work introducing computational complexity theoretic measures of chaos or the ``complexity" of a dynamical system \cite{crutchfield1989inferring,lloyd1988complexity,grassberger2005complexity,cecconi2003complexity,eckmann1985ergodic,stoop2004complexity}.
\item Work studying the way in which a dynamical system, acting as a computer, may solve a decision problem \cite{PhysRevLett.83.1463,cotler2024computational,bruera2024topological,ercsey2011optimization,sudoku2012chaos,bournez2017polynomial}.
\end{enumerate}

Each of the aforementioned papers have taken notable steps in developing the relationship between dynamics and complexity theory, and much of the above work \cite{bournez2017polynomial,PhysRevLett.83.1463,cotler2024computational,bruera2024topological,ercsey2011optimization,sudoku2012chaos,rojas2025algorithmic} has in fact beautifully elucidated fundamental structural relationships between the disciplines.
We also mention that there has been great success in carrying techniques from the sister literature of statistical mechanics into complexity theory  \cite{percus2006computational, gogioso2014aspects, welsh1990computational,galanis2014phase}, as well as a growing and exciting work concerned with using stochastic thermodynamics to better understand energetic aspects of computation \cite{bennett1982thermodynamics, wolpert2019stochastic, kardecs2022inclusive, manzano2024thermodynamics, ouldridge2023thermodynamics,wolpert2024stochastic,PhysRevResearch.2.033312}, and work seeking to develop more general relationships between complexity theory and physics \cite{denef2007computational,denef2018computational}.

Nevertheless, to the best of our knowledge, there is no work rigorously establishing the apparent relationship between degrees of computational feasibility and dynamical complexity, even though such a relationship has long been discussed and pursued (see e.g.\ discussion and questions posed in \cite{bournez2013computability}, and \cite{yampolsky2021towards} for a nice survey of concepts in this area). This apparent gap in the literature in part motivates the program pursued by the present collection of papers.

Our approach spans points (1) and (2) above.
Similar to literature studying the computational complexity of deciding certain problems coming from a dynamical system such as a ``state prediction" or ``reachability problems," we derive a natural class of decision problems coming from dynamical systems.
However, in contrast to previous works studying the complexity of deciding reachability problems, our approach does not restrict attention to a single dynamical system or small sub-class of dynamical systems. Nor are we principally concerned with proving the $\ccNP$ or $\textsf{PSPACE}$-hardness of the decision problems coming from such dynamical systems.
Rather, we emphasize the development of a bidirectional relationship between the fields by showing
\begin{enumerate}
\item how complexity theoretic machinery---namely reductions and an ability to prove lower bounds---provide new techniques for studying and classifying dynamical systems, and
\item how dynamical systems exhibiting certain behaviors impose structure on the algorithms solving decision problems derivable from the dynamics, thereby providing footholds for the application of existing classifications of dynamical systems to be leveraged in computational complexity theory.
\end{enumerate}

There has been literature pursuing programs similar to the spirit of ours. In particular, a nice recent paper of Rojas and Sablik \cite{rojas2025algorithmic} shows deep connections between aspects of computable analysis and dynamical systems, giving a general framework for classifying dynamical systems on the basis of the computational tractability of computing approximations to their attractors. Although the essential idea we employ is similar in spirit to that of Rojas and Sablik, our techniques and results are quite distinct.

In addition, there has been work pursuing a somewhat related aim in the restricted setting of cellular automata and sequential dynamical systems, primarily seeking to classify such systems on a basis of the complexity or decidability of solving certain problems \cite{barrett2006complexity,applebaum2010cryptography,buss1990predictability,SutnerComplexityOfCA,sutner1995computational,sutner2012computational}.
However, this is largely attributed to the discrete nature of cellular automata and sequential dynamical systems, which are amenable to study in the classical complexity setting.
What is lost in this restriction is the rich theory concerning dynamical systems over continuous domains.
In fact, as our work demonstrates, there is nothing to gain but much to lose by restricting to the study of dynamical systems more obviously amenable to discrete analysis, and indeed our results are of a very different nature.

\subsection{Structure of the work}
This is the first of three papers seeking to develop a theory bridging components of computational complexity with dynamical systems. The papers are divided on the basis of the model of computation they work with and the field they emphasize. The present paper deals with the Turing model and emphasizes applications in dynamical systems. The second part \cite{everett2} deals with algebraic models of computation and deals with complexity theoretic issues. The third paper \cite{everett3} works with underlying ideas similar to the present paper and focuses on applications in dynamical systems as well. However, the model of computation is entirely distinct, and breaks free from the difficulties attributed to working in a Boolean framework, thereby providing results of a different nature.

This paper is organized as follows. After a preliminaries section, in \cref{secDiscretization} we develop the crucial notion of efficiently discretizable dynamical systems.
\cref{secTelic} introduces telic problems in full generality for the Turing model of computation, and gives basic structural results.
\cref{secClassification} shows how ability to prove lower bounds can be used to show a dynamical system possesses certain properties, proving \cref{thm:entropyHardness1} and \cref{thmTopDichotomy1}.
Finally, \cref{secProblems} offers concluding remarks and open problems.

\subsection*{Acknowledgements}
I would like to thank David Cash for the conversations and support throughout the rather long development of this project.
I would also like to thank Jordan Cotler for the feedback on the ideas of this paper, and Sam Freedman for input on the exposition of the paper.
Finally, I would like to thank Nicolai Haydn for providing feedback during the early stages of this project that helped shape its direction.

The author is partially supported by the National Science Foundation Graduate Research Fellowship Program under Grant No. 2140001.

\section{Preliminaries}\label{secPreliminaries}

We provide a minimal review of the concepts necessary for the development of this paper, and set notation. Because this paper is intended for a cross-disciplinary audience, we give many basic notions that may be skipped by the expert. For a more complete review of the basics of computational complexity theory see \cite{arora2009computational} or \cite{papadimitriou}. For proper treatments of dynamical systems theory, we suggest \cite{robinson1998dynamical},\cite{brin2002introduction}, and \cite{walters2000introduction}.

\subsection{Dynamical systems background}

For the purpose of this paper a \emph{dynamical system}, or simply a \emph{system}, is a pair $(X, T)$ where throughout $X\subseteq \RR^m$ is infinite and carries the subspace topology induced by the Euclidean metric $\rho$, and $T:X \rightarrow X$ is a transformation. Define the iterates $T^n:X \rightarrow X$ for every nonnegative integer $n$ to be the $n$-fold composition of $T$; if $T$ is invertible $n$ may be negative as well.
When we say $(X, T)$ is a \emph{topological dynamical system} (suppressing notation of the topology), we mean $X$ is a compact and connected subset of $\RR^d$, and $T$ is continuous. If $T$ is also invertible it is taken to be a homeomorphism.
When we say $(X, T)$ is just a \emph{dynamical system} we mean $X$ need not be compact or connected and $T$ need not be continuous.

We now give a number of basic definitions in dynamics needed for this paper.
A subset $A \subseteq X$ is said to be \emph{invariant} if $T(A) = A$, and \emph{forward invariant} if $T(A) \subseteq A$. A point $x \in X$ is said to be \emph{periodic} if there exists an integer $n \geq 1$ such that $T^n(x) = x$.
Two continuous maps $T:X \rightarrow X$ and $S:Y \rightarrow Y$ are said to be \emph{topologically conjugate}, or just \emph{conjugate}, if there exists a homeomorphism $h:X \rightarrow Y$ such that $h\circ T = S \circ h$. We say $h$ is a \emph{topological semi-conjugacy} from $T$ to $S$ provided $h$ is continuous, onto, and $h\circ T = S \circ h$; $S$ is then said to be a \emph{factor} of $T$.

A map $T:X \rightarrow X$ is \emph{topologically transitive} on an invariant set $Y$ provided the forward orbit of some point $p$ is dense in $Y$. In particular, by the Birkhoff Transitivity Theorem, $T$ is transitive on $Y$ if and only if given any two open sets $U$ and $V$ in $Y$, there is a positive integer $n$ such that $T^n(U) \cap V \neq \emptyset$.

Finally, we recall the definition of the topological entropy of a system. We give the definition of Bowen and Dinaburg \cite{bowen1971entropy,dinaburg1970correlation}.
Let $X$ be a compact metric space and $T:X\rightarrow X$ continuous.
For each $n \in \NN$ define a metric $\rho_n$ by
\[
\rho_n(x, y) = \max\{\rho(T^i(x),T^i(y)): 0\leq i < n\} \text{ for } x, y \in X.
\]
For $\varepsilon > 0$ and $n \geq 1$, two points $x, y \in X$ are said to be $\varepsilon$-close with respect to metric $\rho_n$ if their first $n$ iterates are $\varepsilon$-close.
We shall say $E \subset X$ is $(n, \varepsilon)$-separated if each pair of distinct points of $E$ is at least $\varepsilon$ apart in the metric $\rho_n$. Let $N(n,\varepsilon)$ denote the maximum cardinality of an $(n,\varepsilon)$-separated set.
The topological entropy is then defined as
\[
h(T) = \lim_{\varepsilon\rightarrow 0}\left(\limsup_{n\rightarrow \infty}\frac{1}{n} \log N(n,\varepsilon)\right) \geq 0.
\]
The topological entropy can be interpreted as a measure of the amount of randomness in a system. We refer the computer scientist demanding more intuition to e.g.\ \cite{walters2000introduction}.

\subsection{Computational complexity background and notation}

If $x$ is a rational number, let $\langle x \rangle$ denote the standard binary encoding of $x$. More generally, $\langle \cdot \rangle$ denotes the binary encoding of an object. Let $1^n$ label the unary encoding of $n$ as a string of $n$ ones. Let $\{0,1\}^*$ denote the collection of all binary strings, and $\{1\}^*$ denote the collection of all unary strings. We let $|s|$ denote the length of a binary string $s \in \{0,1\}^*$. 

Informally, a \emph{Turing machine} is a system with a collection of rules dictating how one configuration of the system is to be carried over to another configuration.

\begin{definition}
A \emph{Turing machine $\ccM$} is described by a tuple $(\Gamma, Q, \delta)$, where $\Gamma$ is a finite set of symbols $\ccM$'s tape can contain, with designated ``blank" and ``start" symbols. $Q$ is a finite set of states $\ccM$ can be in, with designated starting state $q_{start}$, and halting states $q_{accept}$ and $q_{reject}$. The transition function \[\delta:Q \times \Gamma \rightarrow Q \times \Gamma \times \{L, S, R\}\] describes the rules $\ccM$ uses in performing each step of the computation. It is useful to consider the transition function as a \emph{finite} table that establishes a correspondence between what the input symbol, work tape symbol, and current state is, with the next action.
\end{definition}

Turing machines so defined are also called \emph{deterministic Turing machines.}
The collection of strings that a Turing machine $\ccM$ accepts is the \emph{language of $\ccM$}, denoted $L_\ccM$. A Turing machine that halts on every input, thereby rejecting or accepting every input string, is called a \emph{decider}. A language (a set of strings) is said to be \emph{decidable} or \emph{computable} if some Turing machine decides it --- the input strings the Turing machine accepts are precisely those in the language. A \emph{decision problem} is a computational task that asks for a simple yes-or-no answer about whether a given input satisfies a specified property, and hence can be identified with a language $L \subseteq \{0,1\}^*$ that is precisely the set of strings a decider accepts.

A Turing machine $\ccM_f$ \emph{computes} a partial function $f$ if for all inputs $x$ on which $f$ is defined, $\ccM_f(\langle x\rangle)$ halts with $\langle f(x)\rangle$ on its output tape. A sequence of functions $\{f_n\}_{n=1}^\infty$ is \emph{uniformly computable} if there is a Turing machine $\ccM$ that on any input $\langle n, x\rangle$ halts with $f_n(x)$ on its output tape.

Now we formalize the notion of \emph{runtime} of a Turing machine. We call a function $q:\NN\rightarrow \NN$ \emph{time constructible} if $q(n) \geq n$ and there is a Turing machine that computes the function $x\mapsto \langle q(|x|)\rangle$ in time $q(n)$.

\begin{definition}
Let $\ccM$ be a Turing machine, $f:\{0, 1\}^* \rightarrow \{0, 1\}^*$ and take $q:\NN \rightarrow \NN$ to be  time-constructible. We say that $\ccM$ \emph{computes $f$ in $q$-time} if its computation on every input $x$ requires at most $q(|x|)$ steps.
\end{definition}

The set $\textsf{DTIME}(q(n))$ for time constructible $q:\mathbb{N}\rightarrow \mathbb{N}$, is the class of problems solvable by a (deterministic) Turing machine in time $O(q(n))$.
Efficient computation is equated with polynomial running time. The complexity class $\ccP$ is defined to be $\textsf{P} = \cup_{c \geq 1} \textsf{DTIME}(n^c)$.

The complexity class $\ccNP$ is the class of problems with \textit{efficiently verifiable solutions}. 
That is, a language (decision problem) $L \subseteq \{0, 1\}^*$ is in $\ccNP$ if there exists a polynomial $p:\NN\rightarrow \NN$ and a polynomial time Turing machine $\ccN$ called the \emph{verifier} for $L$, such that for every $x \in \{0, 1\}^*$, 
\[
x \in L \iff \exists u \in \{0, 1\}^{p(|x|)} \text{ s.t. } \ccN(x, u) = 1 \text{ (accepts).}
\]
If $x \in L$ and $u \in \{0, 1\}^{p(|x|)}$ satisfy $\ccN(x, u) = 1$, then we call $u$ a \emph{certificate} for $x$ with respect to the language $L$ and machine $\ccN$.

Let $A$ and $B$ be decision problems (languages). A \emph{polynomial-time many-one reduction} or \emph{Karp reduction} from $A$ to $B$ is a polynomial-time algorithm transforming instances $x$ of $A$ into instances $f(x)$ of $B$, such that $x \in A \iff f(x) \in B$. If such a function exists we write $A \leq_p B$. A problem $A$ in $\ccNP$ is $\ccNP$-complete if every problem in $\ccNP$ is Karp reducible to $A$.

\section{Discretizing dynamical systems}\label{secDiscretization}

This section deals with the basic issues involved in discretizing dynamical systems with continuous state-spaces. We comment that although it may appear most natural to use an algebraic model of computation or Blum-Shub-Smale machines that handle real numbers directly, even in these settings we quickly run into issues of function approximation, so little is lost by remaining in the Turing model.

\subsection{Notation}
Let $\DD$ denote the set of \emph{dyadic rational numbers}; that is, the numbers with finite binary expansion so that $y \in \DD$ has form $y = p/2^q$ for some integers $p$ and $q \geq 0$. The binary expansion of a dyadic rational is its natural finite representation. More precisely, let $s_0, s_1,\dots,s_u, t_1,\dots,t_v$ be bits in $\{0,1\}$. Then the string $s_us_{u-1}\cdots s_0.t_1t_2\cdots t_v$ is the (nonnegative) dyadic rational number
\[
y = \sum_{i=0}^u s_i \cdot 2^i + \sum_{j=1}^v t_j\cdot 2^{-j}.
\]
Let $\langle y \rangle$ denote the binary representation of $y \in \DD$.
In this paper we pass freely between the representation of a dyadic rational $y \in \DD$ by a fraction $p/2^q$ and its associated binary representation. Tacitly, the Turing machine involved deals with the associated binary representation.

Every dyadic rational has an infinite number of representations; when speaking of the dyadic rational $p/2^q$ as a binary string, we mean its shortest representation.
For each representation $s$ of a dyadic rational $d$, let $r = \precision(s)$ label the \emph{precision} of the representation, namely the number of bits to the right of the binary point of $s$. Let $\DD_r = \{m \cdot 2^{-r} : m \in \ZZ\}$ denote the collection of dyadic rationals with precision $r$.

The choice of using the set $\DD$ rather than $\QQ$ comes from the fact that $\DD$ is \emph{uniformly} dense in $\RR$ in the sense that for any integer $r>0$, dyadic rationals of precision $r$ are uniformly distributed on the real line. This is helpful in developing the following definitions.

\subsection{Discretizing state space}

Many canonical dynamical systems have positive Lyapunov exponents.
This diminishes our ability to accurately iterate dynamical systems in a classical (finite) model of computation with fixed precision, where round-off errors lead to exponential decay in the information of a points location in state space. This makes accurate long-term simulation of such systems difficult or impossible (this is a well studied area, with some positive results \cite{gora1988computers}).

In this paper, however, we are only concerned with asymptotics. As a consequence, we do not fix the precision but instead allow it to be sent to infinity with the number of iterations of the dynamical system. This enables accurate simulation of systems with positive Lyapunov exponents.
Our approach toward discretizing dynamical systems is novel (to the best of our knowledge), but should feel relatively standard. Indeed, it is close in spirit to the ``bit-model" of computation \cite{braverman2006computing}, and also reflects aspects of methods by Blank \cite{blank1997discreteness} and Ko \cite{ko2012complexity}, as well as Galatolo in \cite{galatolo2003complexity}, who studies systems that are of a similar spirit to efficiently discretizable systems, albeit in a setting forgoing any notion of efficient computation.

Let $(X, T)$ be a dynamical system. Recall $\DD_r^d \subset \RR^d$ denotes the precision $r$ dyadic rationals. Define the \emph{$r$-discretization} of state space $X\subseteq \RR^d$ to mean the selection of a subset $X_r \subseteq \DD_r^d$ such that for every $x \in X$ there is at least one $x_r \in X_r$ so that $\|x - x_r\|_2 \leq 1/2^r$.
By the \emph{operator of an $r$-discretization} we shall mean the mapping $\ccD_r:X\rightarrow X_r$ that takes each point $x \in X$ to the closest point in $X_r$ with respect to the 2-norm (if there are several such points, select the lowest value).
If $A \subseteq X \subseteq \RR^d$, then $A_r = \ccD_r(A)$ denotes the $r$-discretization of $A$.
If it can be decided in time polynomial in $r$ (encoded as a unary input) whether a point $x_r \in \DD_r^d$ is contained in $X_r$, we shall say the $r$-discretization of $X$ is \emph{efficiently discretizable}.

We say a space $X$ \emph{contains its $r$-discretizations} if $X_r \subseteq X$ for every $r \in \NN$. And, whenever we subscript a space $X$ with an integer, we mean the discretization at a precision by that integer.

\begin{remark}
In the setting of this paper, the operator $\ccD_r$ is simply the rounding operator, which is computable in polynomial time.
\end{remark}

Before moving to define the discretization of a dynamical system, we consider the discretization of a function. First, let
\[
X_* = \bigcup_{r=1}^\infty X_r
\]
denote the union of all discretizations of a state space $X$.

The following definitions are technical because it is necessary to work with Turing machines to maintain precision. Although the ideas are simple: the definitions define a function to be \emph{efficiently discretizable} if there exists a polynomial-time Turing machine capable of computing arbitrarily good approximations to it.

\begin{definition}\label{defFuncDisc}
Let $X, Y \subseteq \RR^d$ be efficiently discretizable, and let $\ccD_r^X, \ccD_r^Y$ be the associated discretization operators for the two spaces. We say a function $f:X\rightarrow Y$ is \emph{efficiently discretizable} if there exists an algorithm (called its \emph{discretization}) $\ccM_f:X_*\times \{1\}^* \rightarrow Y_*$ computable in polynomial time, such that for $x \in X_*$, $\ccM_f(\langle x\rangle, 1^l) = \ccD^Y_l(f(y))$, where $y=x$ if $x \in X$ and $y \in (\ccD_r^X)^{-1}(x)$ otherwise.
\end{definition}

\begin{definition}\label{defFuncDisc2}
If $f:X\rightarrow X$ is a bijection, then we shall say the function is \emph{efficiently discretizable with inverse} if there are uniformly and polynomial-time computable families of functions $\{\ccM_f^{(r)}\}_{r=1}^\infty$, $\ccM_f^{(r)}:X_r\rightarrow X_r$, and $\{\ccM_{f^{-1}}^{(r)}\}_{r=1}^\infty$, $\ccM_{f^{-1}}^{(r)}:X_r\rightarrow X_r$, such that the following hold.\footnote{In other words, there is a polynomial-time Turing machine $\ccM_f(\langle r, x\rangle)$ that rejects its input $\langle r,x\rangle$ whenever $r \neq |\langle x\rangle|$, and we use the notation $\ccM_f^{(r)}(x)\coloneqq\ccM_f(\langle r, x\rangle)$.}
\begin{enumerate}
\item For all $r \geq 1$ and $x \in X_r$, $\ccM_f^{(r)}(\langle x \rangle) = \ccD^X_r(f(y))$, where $y=x$ if $x \in X$ and $y \in (\ccD_r^X)^{-1}(x)$ otherwise; similarly for each $\ccM_{f^{-1}}^{(r)}$, mutatis mutandis.
\item For all $r \geq 1$, $\ccM^{(r)}_f \circ \ccM^{(r)}_{f^{-1}} = \ccM^{(r)}_{f^{-1}} \circ \ccM^{(r)}_{f} = \Id$.
\end{enumerate}
Fix the notation
\[
f_r \coloneqq \ccM^{(r)}_f \quad \text{and}\quad  f_r^{-1} \coloneqq \ccM^{(r)}_{f^{-1}}.
\]
\end{definition}

We now provide our approach for discretizing dynamical systems.

The key idea to keep in mind throughout the development of telic problems and the following definitions is that whenever we wish to iterate the dynamics for a longer (but finite) period of time, this will be matched by an increase in the accuracy of the approximation of the state-space by considering a larger $r$ in the $r$-discretization of $X$, so that the discretized version of the dynamics accurately reflects the true dynamics with respect to the given measurement precision, over the given time period. This is the same as adding more precision to the computer whenever we wish to simulate the dynamics for a longer period of time.

The following linchpin definition formalizes the concept of ``accurately" and ``efficiently" simulating a dynamical system for finite time, measured through some coarse-graining of state space (namely, an $r$-discretization).

\begin{definition}\label{defDynDisc}
Let $X \subset \RR^d$ be efficiently discretizable. We say a dynamical system $(X, T)$ is \emph{efficiently discretizable} if there exists a polynomial-time Turing machine
\[\ccM_T:X_*\times \{1\}^* \times \{1\}^* \rightarrow X_*\]
such that on input $(x, 1^k,1^r)$,
\[
\ccM_T(\langle x\rangle, 1^k, 1^r)=\ccD_r(T^k(y)),
\]
where $y=x$ if $X$ contains its $r$-discretizations, and $y \in \ccD_r^{-1}(x)$ otherwise.
We say $\ccM_T$ is the \emph{discretization of $T$}. Set the notation
\[
T_r^k(x) \coloneqq \ccM_T(\langle x\rangle, 1^k, 1^r) \text{ for $x \in X_*$}.
\]
If $(X, T)$ is a dynamical system with $T$ a bijection that is efficiently discretizable with inverse, then we say $(X, T)$ is \emph{efficiently discretizable with inverse}. We shall say \emph{$(X, T)$ is a topological dynamical system with inverse} to mean $X$ is compact and $T$ is an efficiently discretizable homeomorphism with inverse.
\end{definition}

\subsection{Discretizations of dynamical systems in practice}\label{subsecMapDiscretization}

Intuitively, $T$ is efficiently discretizable if, upon fixing some working precision $r$ for a coarse-graining of the state space $X$, there is some finer precision on which we can simulate the dynamical system for $k$ steps, so that its itinerary in the partition of the coarse-graining is identical to the true length $k$ orbit segment of the system. If we define instead $T_r = \ccD_r \circ T$ we obtain the implementation of the system on a computer with finite precision where points are subject to round-off error at each iteration. In contrast, in our setting we round ``only on the last iteration."

It follows that if $T$ is the doubling map $E_2(x) = 2x \mod 1$ of the circle for instance, although distances of nearby points double every iteration, there is an appropriate precision $r=r(k) \geq 2k$ so that $\|T^{k}(x) - T_r^{k}(x)\|_2 \leq 1/2^{r+1}$ for all $x \in X_r=I_r$ ($X=[0,1)$ which contains its $r$-discretization). That is, the discretization $\ccM_T$ iterates $x$ at a precision of $r=r(k)\geq 2k$ for $k$ iterations, and even with $k$ worst-case round-off errors each iteration, we still have 
\[
\|T^{k}(x)-T_r^{k}(x)\|_2\leq 2^{-k+r} \leq 1/2^{r+1}.
\]
Since $k$ is polynomially related to $r$, the doubling map is then seen to be efficiently discretizable, since
\[
\ccM_T(\langle x\rangle, 1^k, 1^r)=\ccD_r(T^k(x))
\]
for all $x \in X_r$. From this argument it follws that a wide variety dynamical systems whose maps are ``well-behaved" Lipschitz maps are efficiently discretizable. Notice we say ``well-behaved" rather than \emph{all} Lipschitz maps, because they are in general not even computable (in an appropriate sense).

Other basic examples of efficiently discretizable systems include rigid circle rotations $R_a(x) = x+a \mod 1$ by an algebraic number $a \in \AA$,\footnote{There is a mature literature on the efficient and effective manipulation of algebraic numbers that deals with their representations as roots of polynomials with integer coefficients. See \cite{basu2006algorithms}.} the tent-map $T:[0,1]\rightarrow[0,1]$, defined by
\[
T(x) = \begin{cases}
2x &\text{ if } 0\leq x\leq 1/2,\\
2(1-x) &\text{ if } 1/2<x\leq 1,
\end{cases}
\]
as well as the H\'enon map, Baker's map, and Chirikov standard map with rational parameters. In fact ``most" dynamical systems one comes across will be efficiently discretizable since those that are not tend to be rather pathological (i.e. uncomputable examples). For completeness, we prove this assertion for both rigid circle rotations and the tent map in the following lemma.

\begin{lemma}\label{lemEfficientDiscretization} Let $R_a:S^1\rightarrow S^1$ be a rigid rotation of the circle by an algebraic number $a \in \mathbb{A}$, and let $T:I\rightarrow I$ be the tent-map.
\begin{itemize} 
\item The system $(S^1, R_a)$ is efficiently discretizable.
\item The system $(I, T)$ is efficiently discretizable.
\end{itemize}
\end{lemma}
\begin{proof}
Put $S^1 = \RR/\ZZ$, so treat $R_a$ as a map of the unit interval $I$ with endpoints identified. For every $r \in \NN$, points in the $r$-discretization $I_r$ of $I$ have form $w/2^r$, $w = 0,\dots,2^{r}-1$. Notice for any $x=\frac{w}{2^r} \in I_r$, $R_a^k(x) = w/2^r + ka \mod 1$. But $w/2^r + ka$ can be computed and reduced mod $1$ in time polynomial in $r +k$ when $a \in \AA$ is algebraic, and $\ccD_r$ is simply a rounding operation.

For the system $(I, T)$, again notice that for any $x \in I_r$, $x$ is of form $w/2^r$. Then, either $T(x) = 2w/2^r$ or $T(x) = 2^{r+1}/2^r - 2w/2^r = w'/2^r$ with $w' \in [2^r-1]$.
But notice if $x \in I_r$, then so too is $T(x) \in I_r$. Therefore $T^k(x) \in I_r$, which is computable in time polynomial in $r+k$, as is the computation of $\ccD_r(\cdot)$.
\end{proof}

\section{Finite telic problems}\label{secTelic}

\subsection{Telic problems and first results}

We begin with a preliminary definition. Intuitively, a map $g:I^d_i\rightarrow X_{r_i}\cup \{\perp\}$ satisfies the \emph{neighborhood preservation property} if $g(I^d_i)$ is a distortion of $I^d_i$ in $X$, keeping adjacent points in $I^d_i$ adjacent in $X$.

\begin{definition}\label{defFeasibleQuery}
Fix an efficiently discretizable space $X \subset \mathbb{R}^\ell$ and an integer $d$, $1 \leq d \leq \ell$. Let $I_r^d$ be the $r$-discretization of the $d$-dimensional cube $[0,1]^d \subset \RR^d$.
Let $\{g^{(i)}\}_{i=1}^\infty$ be a sequence of maps
\[
g^{(i)}:I^d_i\rightarrow X_{r_i} \cup \{\perp\},
\]
where $r_i\geq i$ is a value determined by $g^{(i)}$.
We shall say the maps satisfy the \emph{neighborhood preservation} property if for all $i \in \NN$, it holds that for every $x_0 \in I^d$ and $\varepsilon=\sqrt{d}/2^{i}$, there is a $\delta > 0$ and $y_0 \in X$ such that for all $x\in I_i^d $, $x \in \overline{B_\varepsilon(x_0)}$ if and only if $g^{(i)}(x) \in \overline{B_\delta(y_0)} \cup \{\perp\}$.\footnote{The intuition is that $g^{(i)}$ distorts the structure of $I_i^d$ in $X$, but not too dramatically so that neighboring points in $I_i^d$ are still neighboring in $\image(g^{(i)})$.}
If, in addition, there is a Turing machine $\ccG$ running in time $O(i^c)$, $c \in \NN$, for which $\ccG(\langle x\rangle) = \langle g^{(i)}(x)\rangle$ where $|\langle x\rangle|=di$, then we shall say the sequence $\{g^{(i)}\}_{i=1}^\infty$ is \emph{feasibly queryable}.
\end{definition}

The symbol $\perp$ in the definition should be interpreted as a ``blank" or ``garbage" output of $g^{(i)}$. The $g^{(i)}$ can be interpreted as maps from binary strings to $r$-discretizations of state-space, i.e.\ $g^{(i)}:(\{0,1\}^i)^d \rightarrow X_{r_i} \cup \{\perp\}$.
The neighborhood preservation property is quite strong, and is used in the following definition of telic problems to ensure any difficulty in deciding telic problems comes from the dynamics, not complicated functions $g^{(i)}$.

The following definition of telic problems is necessarily technical in order to make the computational complexity theoretic claims of this paper precise and correct. We give intuition after the definition.

\begin{definition}[Telic problems]\label{deftelic}
Fix $X \subseteq \RR^m$, and let $(X, T)$ be an efficiently discretizable dynamical system with $\ccM_T$ the discretization of $T$.
\begin{itemize} 
\item Let $\{g^{(i)}\}_{i=1}^\infty$ be a sequence of feasibly queryable maps, with
\[
g^{(i)}: I_i^d \rightarrow X_{r_i} \cup \{\perp\}
\]
for $d \leq m$.
Let $H:X\cup\{\perp\}\rightarrow X\cup\{\perp\}$ be a map such that $H$ carries $\perp$ to itself, and $H|_X\coloneqq \hat{H}:X\rightarrow X$ is an efficiently discretizable bijection with inverse that commutes with $T$, so $H\circ T = T \circ H$ ($H$ is typically the identity).
Fix $\ell \in \ZZ^+$.
Let $\{h^{(i)}\}_{i=1}^\infty$ be a family of maps with $h^{(i)} = \hat{H}_{r_i}^{\ell}\circ g^{(i)}$, and let $\ccH$ be a Turing machine running in polynomial time, with $\ccH(\langle x\rangle) = \langle h^{(i)}(x)\rangle$ for $x \in I_i^d$.

\item Let $\ccV:\{1\}^*\rightarrow \{1\}^*$ denote a polynomial-time Turing machine that on input $1^n$ returns $1^{Cn}$ for some integer $C\geq 1$.
Let $v(n)=Cn$ be the integer representation of the unary value $\ccV(1^n)$.\footnote{$\ccV$ is a machine taken to run in polynomial-time and hence can return any polynomial value of $n$ in unary most generally. Working at this level of generality however grants us nothing of use in this paper so we do not work with this for now.}

\item Let $\{B^{(i)}\}_{i=1}^\infty$ be a sequence of sets treated as subsets of ambient space $\RR^m$ of $X$, for which $B^{(i)} = \hat{H}^{\ell}(R^{(i)})\cap X\neq \emptyset$ where $R^{(i)}\subset \RR^m$ is a closed rational rectangle with non-empty intersection with $X$, and there is an associated polynomial time algorithm $\ccB$ that, on input $x \in X_{v(i)}$, returns $1$ if $x \in\ccD_{v(i)}(B^{(i)})$ and returns $0$ otherwise. Fix the notation $B^{(i)}_{v(i)} = \ccD_{v(i)}(B^{(i)})$.
\end{itemize}

A \emph{telic problem} coming from an efficiently discretizable dynamical system $(X, T)$,  is the problem of deciding membership in the unary language
\begin{equation*}
\ccTELIC_{(X, T)}(\ccV,\ccH,\ccB) =\left\{1^n : \exists s \in I_n^d \text{ s.t. } \ccB(\ccM_T(\ccH(\langle s\rangle), 1^n, \ccV(1^n))) = 1\right\}.
\end{equation*}
Let $\ccTELIC_{(X, T)}$ denote the entire class of telic problems coming from $(X, T)$.
\end{definition}

We often use the following notation throughout this paper when denoting telic problems, in order to make all the variables and objects clear from setup.
\begin{equation*}
\ccTELIC_{(X, T)}(v,\{h^{(i)}\},\{B^{(i)}\}) =\left\{1^n : \exists s \in I_n^d \text{ s.t. }T_{v(n)}^{n}(h^{(n)}(s)) \in B^{(n)}_{v(n)}\right\}.
\end{equation*}

Figure \ref{figTelic} supplies a pictorial representation of the central idea of telic problems.
Some comments explaining the decisions made in the definition are in order.

\begin{figure}
    \centering
    \includegraphics[scale=.19]{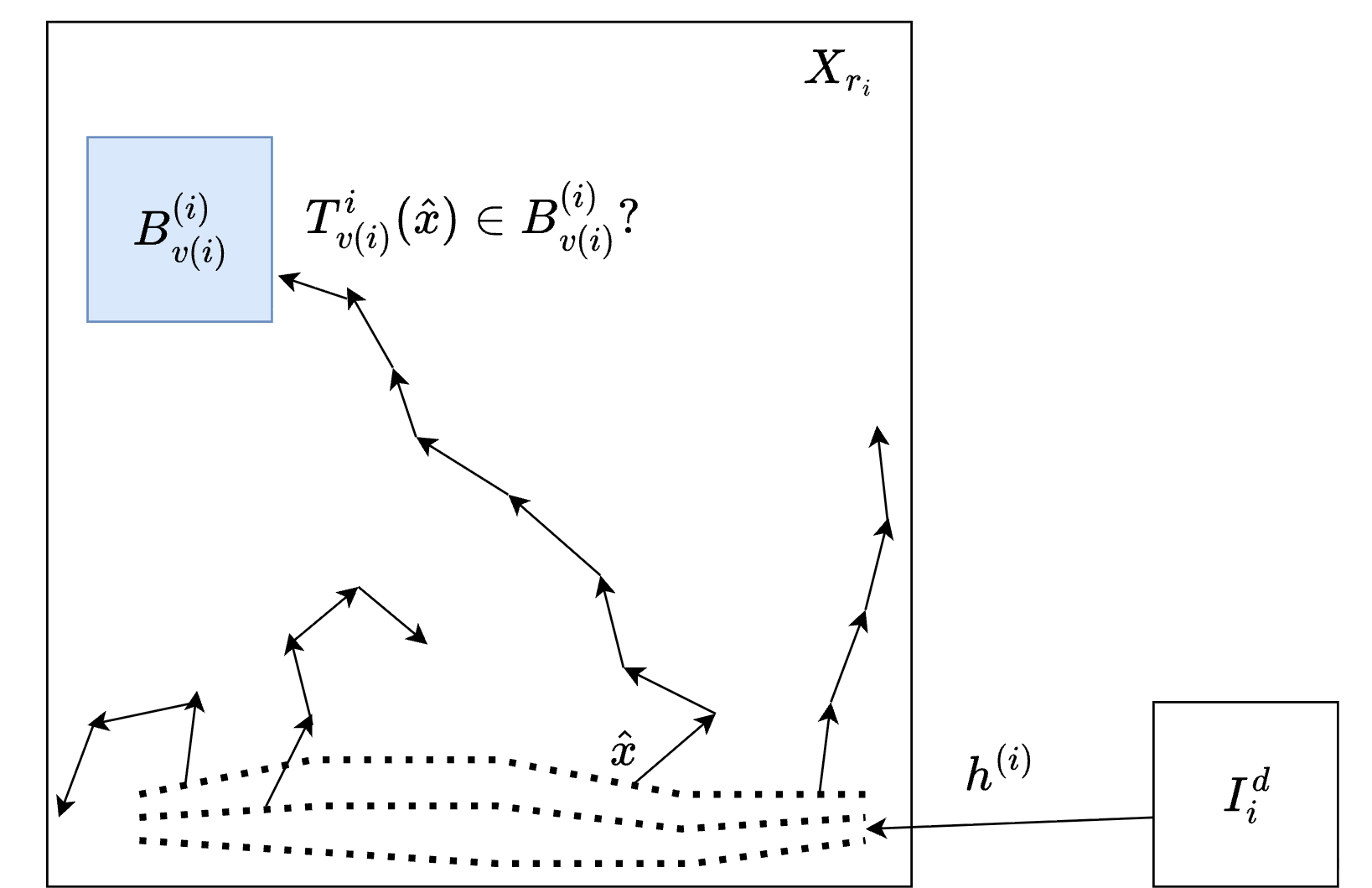}
    \caption{A representation of the central idea of telic problems. A map $h^{(i)}$ carries points from the unit dyadic rationals (tuples of binary strings of length $i$) into points in state space in an orderly fashion. The problem is to decide whether one of the points --- $\hat{x}$ in this case --- evolves to a region nearby or inside of the target set $B^{(i)}_{v(i)}$ after $i$ of iterations of the discretization of the map.}
    \label{figTelic}
\end{figure}

First, we comment that the value $v(n)$ given by $\ccV(1^n)$ is used to define the ``operating precision" for the given instance $1^n$ of the telic problem. That is, it determines the ``resolution" at which we wish to measure the dynamics. Most simply, $\ccV(1^n) = 1^n$, but this makes less sense when $n$ is small, and allowing the operating precision to change provides more generality.

Our reasoning for allowing the maps $g^{(i)}$ to carry points into a finer discretization $X_{r_i}$ rather than the coarser $X_i$, is because the finer precision provides more flexibility in choosing where to iterate initial points from. That is, if $r_i=i$ then $g^{(i)}$ may map $I^d_i$ ``entirely onto" $X_i$, which provides less flexibility in structuring telic problems.

In this paper the function $H$ is always taken to be either the identity or $T$.

Notice we have the following polynomial time algorithm for $\ccB$: since $H$ is an efficiently discretizable bijection with inverse, to decide membership in $B^{(i)}_{v(i)}$, it suffices to check whether
\[\hat{H}^{-\ell}_{r_i}(x)\in \hat{H}^{-\ell}_{r_i}(B^{(i)}_{v(i)}),\]
since $\hat{H}^{-\ell}_{r_i}(B^{(i)}_{v(i)})$ is the discretization of a closed rectangle where deciding membership is easy. Specifically, the condition that $H$ commute with $T$ is used so that solutions to a telic problem before and after applying $H$ are identical, since
\[
H \circ T = T \circ H \implies H \circ T^k = T^k \circ H \text{ for all $k \in \NN$,}
\]
so $T^k(x) \in A \iff T^k(H(x)) \in H(A)$.

\begin{remark}
Telic problems can be naturally generalized further by varying the number of \emph{iterations} of the map $T$ under consideration. Namely, define a polynomial time Turing machine $\ccK:\{1\}^*\rightarrow \{1\}^*$, and decide membership in the unary language
\begin{equation*}
\ccTELIC_{(X, T)}(\ccV,\ccH,\ccB) =\left\{1^n : \exists s \in I_n^d \text{ s.t. } \ccB(\ccM_T(\ccH(\langle s\rangle), \ccK(1^n), \ccV(1^n))) = 1\right\}.
\end{equation*}
All of the following results remain the same for this more general problem, however we do not consider it here in an attempt at simplifying components of the exposition. In the third part of this work \cite{everett3} we do work at this level of generality and derive useful consequences.
\end{remark}

We give a simple example of a telic problem.

\begin{example}
Consider the logistic map $T(x) = 4x(1-x)$ carrying the unit interval $I$ into itself.
The transformation is efficiently discretizable.
Define a sequence of target sets so that $B^{(n)} = R^{(n)} = [a, b] \subset [0,1]$ with $b-a \leq 1/2^n$ for all $n \geq 1$.
Let $g^{(n)}:I_n\rightarrow I_{n^5} \cup \{\perp\}$\footnote{Choice of $n^5$ is arbitrary} be defined so that $g^{(n)}(x) = x^2$ for all $n\in \NN$, on the inputs $x \in I_n$ for which $g^{(n)}(x) \neq \perp$. The function $x^2$ is feasibly queryable: it satisfies the neighborhood preservation property on $I$, and can be computed by a polynomial-time Turing machine.
In particular, the $g^{(n)}$ preserve the lexicographic ordering of the length $n$ binary strings in $[0,1]$ with respect to $<$.
The logistic map with the given parameter value is ``chaotic," so it would be expected that determining, for each $n$, whether one of the $2^n$ points in $\image(g^{(n)})$ evolves into $B^{(n)}$ is intractable for most choices of target set sequences $\{B^{(n)}\}_{n=1}^\infty$.
\end{example}

Telic problems are deemed to be \emph{easy} or \emph{hard} in the usual way: namely, a telic problem is \emph{easy} or \emph{tractable} if it can be decided by an $O(n^c)$-time Turing machine for $c \in \NN$, and the problem is \emph{hard} or \emph{intractable} otherwise. If a dynamical system has associated easy (hard) telic problems, we say the dynamical system \emph{admits} easy (hard) telic problems.

The following remark is primarily of concern to any complexity theorist reading this paper.

\begin{remark}\label{remTally}
Telic problems are decision problems for \emph{unary} languages. Hence, when discussing the existence of easy and hard telic problems, we mean with respect to a uniform model of computation since $\textsf{TALLY} \subset \ccPpoly$.
However, it is quite reasonable to conjecture there exist telic problems which are intractable in the uniform setting; first, note that telic problems are only in $\ccPpoly$ because of the ability to use advice strings, which in general does not imply there are efficient Turing machines deciding instances of telic problems.
In addition, there exists sparse languages in $\ccNP\setminus \ccP$ if and only if $\textsf{E} \neq \textsf{NE}$ \cite{hartmanis1983sparse}. In particular, if $\textsf{E} \neq \textsf{NE}$ then there is a unary language in $\ccNP\setminus \ccP$. Hence, it is reasonable to argue that there may exist hard telic problems---i.e. $\ccNP$-intermediate telic problems.
\end{remark}

\subsection{Basic properties of telic problems}
Efficient discretizations of dynamical systems faithfully capture the true dynamics in a way described by the following proposition. Informally, iterating a discretization of a system brings points nearby a target set $B^{(i)}_{v(i)}$ if and only if the true dynamics also carry a nearby point into a neighborhood of the target set $B^{(i)}$ in the same number of iterations.
This trivial but important fact allows for correct application of machinery from the theory of dynamical systems to the study of telic problems, since we can be sure an instance of a telic problem has a positive answer if and only if the true underlying system also has a positive answer to the telic problem \emph{before} discretizing.

\begin{proposition}\label{propCorrectSim}
Let $(X, T)$ be an efficiently discretizable dynamical system with $\ccM_T$ the discretization of $T$, and let
\[
\ccTELIC_{(X, T)}(v, \{h^{(i)}\}, \{B^{(i)}\}) \coloneqq \ccTELIC_{(X, T)}(\ccV,\ccH,\ccB)
\]
be some telic problem coming from $(X, T)$. Then for all $n \in \NN$, and any $s \in I_n^d \setminus (h^{(n)})^{-1}(\perp)$, we have
\[
\ccB(\ccM_T(\ccH(\langle s\rangle), 1^n, \ccV(1^{n}))) = 1 \iff \rho(T^{n}(y), B^{(n)}) \leq 1/2^{v(n)},
\]
for some $y \in \ccD^{-1}_{n}(\ccH(\langle s\rangle))$.
\end{proposition}
\begin{proof}
Suppose $\ccB(\ccM_T(\ccH(\langle s\rangle), 1^n, \ccV(1^{n}))) = 1$. Since $\ccM_T$ is the discretization of $T$, we have by definition
\[
\ccM_T(\ccH(\langle s\rangle), 1^n, \ccV(1^{n})) = \ccD_{v(n)}(T^{n}(y)) \in X_{v(n)}
\]
where either $y=\ccH(\langle s\rangle)$ or $y \in \ccD_{v(n)}^{-1}(\ccH(\langle s\rangle))$. So $\ccB(\ccM_T(\ccH(\langle s\rangle), 1^n, \ccV(1^{n}))) = 1$ implies $\ccD_{v(n)}(T^{n}(y)) \in B^{(n)}_{v(n)}$. Hence, by the definition of the operator $\ccD_r$, $\rho(T^{n}(y)), B^{(n)}) \leq 1/2^{v(n)}$, and in particular $T^{n}(y) \in \ccD_{v(n)}^{-1}(B^{(n)}_{v(n)})$.

The converse follows just as easily.
\end{proof}

\begin{lemma}\label{lemInNP}
Telic problems are contained in $\emph{\ccNP}$.
\end{lemma}
\begin{proof}
Let $(X, T)$ be an efficiently discretizable dynamical system with $\ccM_T$ the discretization of $T$, and let $\ccTELIC_{(X, T)}(v, \{h^{(i)}\}, \{B^{(i)}\})$ be some telic problem coming from $(X, T)$.
Presented with an instance $1^n$, we have
\[
1^n \in \ccTELIC_{(X, T)}(v, \{h^{(i)}\}, \{B^{(i)}\}) \iff \exists s \in I_n^d
\]
such that $\ccB(\ccM_T(\ccH(\langle s\rangle), 1^n, \ccV(1^{n}))) = 1$.

But notice $\ccH, \ccM_T$, and $\ccB$ all run in time polynomial in their input lengths, which are at most polynomial in $n$. Hence their compositions do as well.
As such, $s$ serves as a valid certificate and
\[\ccTELIC_{(X, T)}(v, \{h^{(i)}\}, \{B^{(i)}\}) \in \ccNP.\]
\end{proof}

As noted in \cref{remTally}, telic problems as defined in this paper are not $\ccNP$-complete unless $\ccP=\ccNP$. Nonetheless, this fact does not preclude the existence of hard telic problems. Moreover, variations of telic problems can be easily defined which are $\ccNP$-complete, but for the purpose of applying computational complexity theory to dynamical systems considering these more complicated problems offers little to no benefit.

\subsection{Reductions and transformations between telic problems}
When we speak of \emph{reductions} between telic problems, we mean Karp reductions (polynomial time mapping reductions), although Turing reductions work just as well.
If telic problem $\ccT_{(X, T)} \coloneqq \ccTELIC_{(X, T)}(\ccV,\ccH,\ccB)$ coming from $(X, T)$ reduces to telic problem $\ccT_{(Y, S)} \coloneqq \ccTELIC_{(Y, S)}(\ccV',\ccH',\ccB')$ coming from $(Y,S)$, we write $\ccT_{(X, T)} \leq_p \ccT_{(Y, S)}$. In particular, suppose telic problem $\ccT_{(X, T)}$ reduces to $\ccT_{(Y, S)}$ in polynomial time, and $\ccT_{(Y, S)}$ has a polynomial time solution. Then $\ccT_{(X, T)}$ is solvable in polynomial time. And, if $\ccT_{(X, T)}$ is hard, then $\ccT_{(Y, S)}$ is hard.

Recall from the introduction that if $(X, T)$ and $(Y, S)$ are efficiently discretizable topological dynamical systems, and every telic problem coming from $(X, T)$ reduces to a telic problem coming from $(Y, S)$, we shall say $(X, T)$ \emph{reduces} to $(Y, S)$, and write $(X, T)\leq_p (Y, S)$. We then say $(X, T)$ and $(Y, S)$ are \emph{Turing equivalent} if $(X, T) \leq_p (Y, S)$ and $(Y, S) \leq_p (X, T)$, writing $(X, T)\sim_p (Y,S)$.

A primary problem identified in this paper is of determining invariants of Turing equivalence; \emph{if two efficiently discretizable dynamical systems are Turing equivalent, what properties do they share}?\footnote{Note that many of these questions can be formulated in a category-theoretic language if desired; objects are telic problems and the morphisms are Karp reductions. Natural questions then arise about correspondence in structure between such categories and categories of dynamical systems.} Determining a suitable invariant---if one exists---would essentially prove the existence of hard telic problems, since all dynamical systems admitting only easy telic problems are necessarily all Turing equivalent.

As such, we are led to focus on \emph{reductions} between telic problems and the extent to which polynomial time reductions between telic problems imply the associated dynamical systems share certain properties. We posit the positivity of topological entropy is one possible invariant. These ideas are brought together into the following proposition, which demonstrates how understanding the connection between the behavior of a dynamical system and the associated telic problems is related to the $\ccP$ vs.\ $\ccNP$ problem.

\begin{proposition}\label{propReductionEntropy}
Suppose positivity of topological entropy is an invariant of Turing equivalence, i.e.\ $h(T)>0$ if and only if $h(S)>0$ whenever $(X, T) \sim_p (Y, S)$. Then there exists hard telic problems, and in particular, $\ccP\neq \ccNP$.
\end{proposition}
\begin{proof}
Let $T:I\rightarrow I$ be the tent-map, and $R_\alpha:S^1 \rightarrow S^1$ be a rigid rotation of the circle for $\alpha \in \AA$. Then $h(T) = \log2$ and $h(R_{\alpha}) = 0$. In addition, \cref{propTransitivity} demonstrates that $(S^1, R_\alpha)$ does not admit hard telic problems. Hence, assuming the antecedent, there is a telic problem $\ccT_{(X, T)}$ that is not Karp reducible to any telic problem coming from $(S^1, R_\alpha)$. But since $(S^1, R_\alpha)$ admits nontrivial telic problems, this implies $\ccT_{(I, T)} \not\in \ccP$.
\end{proof}

Informally, \cref{propReductionEntropy} asserts that if a system with positive topological entropy is only polynomial time reducible to other systems with positive entropy, then there exists hard telic problems.

A first step in understanding the relationship between reductions and preservation of dynamical properties is to understand the connection with conjugacies, since topological conjugacies of dynamical systems preserve topological invariants, such as entropy.
Given the close connection between telic problems and the dynamics, one would expect that conjugacies can act as reductions in certain instances. The following theorem shows this is indeed the case. Recall an efficiently discretizable state space $X$ contains its $r$-discretizations if $X_r \subset X$ for all $r \in \NN$.

\begin{theorem}\label{thmConjugacyHardness}
Let $X \subset \RR$ be a closed interval, and let $(X, T)$ and $(X, S)$ be two efficiently discretizable dynamical systems.
Suppose $\varphi:X\rightarrow X$ is a conjugacy of the two systems that is efficiently discretizable with inverse.
Then $(X, T)$ admits a hard telic problem if and only if $(X, S)$ admits a hard telic problem.
\end{theorem}
\begin{proof}
Since $\varphi$ is a conjugacy of the dynamical systems it is a homeomorphism, and $\varphi\circ T = S \circ \varphi$. In addition, $\varphi$ is assumed to be efficiently discretizable with inverse.
So, by \cref{defFuncDisc}, \cref{defDynDisc} and the assumption that $X$ contains its $r$-discretizations, we conclude $\varphi_r\circ T_r = S_r \circ \varphi_r$.

Extend $\varphi$ by putting $\hat{\varphi}:X\cup \{\perp\}\rightarrow X\cup\{\perp\}$ so that $\hat{\varphi}(x) = \varphi(x)$ for all $x \in X$ and $\hat{\varphi}(\perp) = \perp$.

Given an instance of a telic problem $\ccTELIC_{(X, T)}(v, \{h^{(i)}\}, \{B^{(i)}\})$ coming from $(X, T)$, the instance can be transformed into a telic problem 
\[
\ccTELIC_{(X, S)}(v, \{f_i\}, \{A^{(i)}\}) = \ccTELIC_{(X, S)}(v, \{\hat{\varphi}_{r_i}\circ h^{(i)}\}, \{\varphi_{i}\left(B^{(i)}_{v(i)}\right)\})
\]
coming from $(X, S)$.
Note that since $h^{(i)} = \hat{H}^\ell_{r_i} \circ g^{(i)}$ for some $\hat{H}$ whose restriction to $X$ is a homeomorphism of the interval $X$, and $\hat{\varphi}$ is also an (efficiently discretizable) homeomorphism when restricted to $X$, it follows that the functions $\hat{\varphi}_{r_i}\circ h^{(i)}$ satisfy the definition of feasibly queryable maps. Furthermore, $\varphi_{i}\left(B^{(i)}_{v(i)}\right) \subset X$ is an interval. Consequently, $\ccTELIC_{(X, S)}(v, \{f_i\}, \{A^{(i)}\})$ forms a valid telic problem.

Because $\varphi_r\circ T_r = S_r \circ \varphi_r$, it follows that $T^{n}_{v(n)}(h^{(n)}(s)) \in B^{(n)}_{v(n)}$ if and only if $S^{n}_{v(n)}(f_n(s)) \in A^{(n)}_{v(n)}$ for all $n \in\NN$.
Hence we have obtained a polynomial time reduction of telic problems coming from $(X, T)$ to telic problems coming from $(X, S)$. The other direction follows by identical argument.
\end{proof}

\begin{remark}
With minor adjustments we may instead we take two systems $(X, T)$ and $(Y, S)$ with $X \neq Y$ and $X, Y \subset \RR$ closed intervals to obtain the same conclusion.
Similarly, by changing the definition of telic problems slightly, the theorem goes through for more general spaces $X \subset \RR^d$.
We do not prove these statements to maintain a simpler exposition.
\end{remark}

The following corollary is an immediate consequence of \cref{thmConjugacyHardness}.

\begin{corollary}
Under the conditions of \cref{thmConjugacyHardness}, if $\varphi$ is an efficiently discretizable semi-conjugacy and $(X, S)$ is a factor of $(X, T)$, then if $(X, S)$ admits a hard telic problem, so does $(X, T)$.
\end{corollary}

It would be desirable to show that if $(X, T)$ admits a hard telic problem while $(Y, S)$ does not, then the systems are not conjugate. However, at the moment we can only conclude there does not exist an \emph{efficiently discretizable} conjugacy of the two systems. One question posed in \cref{secProblems} is whether all conjugacies of efficiently discretizable dynamical systems also efficiently discretizable (Problem \ref{probEfficientConj}). Positive answers to these questions would reveal deep and surprising connections between the disciplines.

\section{Orbit structure and computational complexity}\label{secClassification}

This section begins uncovering the relationship between the behavior of a dynamical system and the complexity of solving its associated telic problems. This relationship furnishes a theory capable of classifying topological dynamical systems by proving lower bounds on the associated telic problems. Moreover, it points toward fundamental connections between dynamical systems and complexity theory, and by extension mathematical logic and definability theory via descriptive complexity theory.

\subsection{First results}

We begin with the following proposition saying topological transitivity is not a sufficient condition for a dynamical system to admit a hard telic problem. The proof can be easily mimicked to prove that other simple systems, such as non-expansive systems, also only admit telic problems contained in $\ccP$. This indicates existence of telic problems coming from a system which are not contained in $\ccP$ implies the system ought to possess properties such as sensitivity to initial conditions, although this is open (see \cref{secProblems}).

\begin{proposition}\label{propTransitivity}
There exists an efficiently discretizable, topologically transitive dynamical system that does not admit hard telic problems.
\end{proposition}

Before proving the statement we give two simple lemmas needed throughout this section.

\begin{lemma}\label{lemEfficientDistCheck}
Let $X \subset \RR^d$ be efficiently discretizable, and let $B \subset X$ be a composed of a union of closed rational rectangles. Then, for all $r \in \NN$ and $x \in X_r\setminus \ccD_r(B)$, the distance $\rho(x, \ccD_r(B))$ (expressed as an algebraic number) is computable in time $\poly(r)$.
\end{lemma}
\begin{proof}
If $B$ is simply a single rational rectangle, for $x \in X_r\setminus \ccD_r(B)$, $\rho(x, \ccD_r(B))$ is computable in polynomial time by taking the distances between $x$ (which has length $r$), and the faces composing $B$. Hence, suppose $B$ is the union of $q$ distinct closed rational rectangles $B = \cup_i B^{(i)}$. Then for each $B^{(i)}$, compute $\rho(x, \ccD_r(B^{(i)}))$ (as an algebraic number, see \cite{basu2006algorithms}), and return the smallest such value. Since $q$ is fixed, the algorithm runs in time $\poly(r)$.
\end{proof}

\begin{lemma}\label{lemEfficientSetDecidability}
Let $X \subseteq \RR^d$ be an efficiently discretizable space, and let $\{g^{(i)}\}_{i=1}^\infty$ be a sequence of feasibly queryable maps, with $g^{(i)}:I_i^d \rightarrow X_{r_i} \cup \{\perp\}$.
For any subset  $A \subset X$ composed of a finite union of rational closed rectangles, and any $n \in \NN$, there exists an algorithm that, on input $1^n$ returns an $s \in I_n^d$ such that $g^{(n)}(s) \in \ccD_{r}(A)$ for $r \geq n$, and $\perp$ otherwise, that runs in time $\poly(n+|\langle A\rangle |)$.
\end{lemma}
\begin{proof}
Since each of the maps $g^{(i)}$ is feasibly queryable they also satisfy the neighborhood preservation property (cf.\ \cref{defFeasibleQuery}).

Divide $I^d$ into $k=2^d$ cubes $C_j$, $j=1,\dots,k$ with the dyadic rational $a_i$ marking its center, $a_j \in I_i$, $i\geq k+1$.
On input $1^n$, $n \geq k+1$, we are to perform a $k$-ary search to decide whether there is an $s \in I^d_n$ such that $g^{(n)}(s) \in \ccD_{r}(A)$. This is straightforward: check the distances of all transformed $k$ cube centers $g^{(n)}(a_j)$ from $\ccD_{r}(A)$ (which is a polynomial time operation by \cref{lemEfficientDistCheck}). Choose the cube with center $g^{(n)}(a_j)$ closest to $\ccD_r(A)$, then recurse, dividing cube $C_j$ into $k$ more cubes: this can be correctly performed on account of the neighborhood preservation condition on the feasibly queryable maps $\{g^{(n)}\}$, which ensures that for for all $n \in \NN$, it holds that for every $x_i \in I^d$ and $\varepsilon = \sqrt{d}/2^n$, there is a $\delta>0$ and a $y_i \in X$ such that for all $x \in I_n^d$, $x \in \overline{B_\varepsilon(x_0)}$ if and only if $g^{(n)}(x) \in \overline{B_\delta(y_0)} \cup \{\perp\}$.

Each iteration the number of candidate solutions is multiplied by a factor of $1/k$, $k \geq 2$. As such, the algorithm halts in time polynomial in $n$ and the length of the description of $A$.
\end{proof}

We now prove \cref{propTransitivity}, showing one side of how transitivity is related to tractability.

\begin{proof}[Proof of \cref{propTransitivity}]
Let $(S^1, R_\alpha)$ be a rigid rotation of the circle by an irrational algebraic number $\alpha \in \AA$ (again, see \cite{basu2006algorithms} as a starting point for a reference on the efficient and effective manipulation of algebraic numbers). \cref{lemEfficientDiscretization} tells us the system is efficiently discretizable, and irrational rotations of the circle are topologically transitive. In particular, notice that the inverse $R_\alpha^{-1}$ also has an efficient discretization.

Let $\ccTELIC_{(S^1, R_\alpha)}(v, \{h^{(i)}\},\{B^{(i)}\})$ be any telic problem coming from the system. We treat $S^1$ as $\RR/\ZZ$, i.e.\ the unit interval $I$ with $1$ and $0$ identified. Hence, since the $g^{(i)}$ are feasibly queryable and thus neighborhood preserving, the dyadic rationals of $I_i$ at precision $i$ are carried into $\RR/\ZZ$ in a manner that preserves their order with respect to the usual ordering $\leq$. Moreover, note that since $h^{(i)} = \hat{H}_{r_i}^m \circ g^{(i)}$ and $\hat{H}^m|_I$ is a homeomorphism, $h^{(i)}$ also carries the dyadic rationals of $I_i$ into $I$ in a way that preserves their order (and are neighborhood preserving). 

One way to solve this telic problem is by effectively running a binary search, using \cref{lemEfficientDistCheck}. We proceed in another way: on input $1^n$, find the endpoints $x_0, x_1$ of the target set $B^{(n)}_{v(n)}$. Let
\[
x_0' = \ccM_{R_\alpha^{-1}}(x_0, 1^{n}, \ccV(1^n)) \in I_n \quad \text{and}\quad x_1' = \ccM_{R_\alpha^{-1}}(x_1, 1^{n}, \ccV(1^n)) \in I_n.
\]
Put
\[
\hat{B}^{(n)}_{v(n)} = [\hat{H}_{r_i}^{-m}(x_0'), \hat{H}_{r_i}^{-m}(x_1')].
\]
Using \cref{lemEfficientSetDecidability}, check if there is an $s \in I_n$ such that $g^{(n)}(s) \in \hat{B}^{(n)}_{v(n)}$. Respond in kind.
\end{proof}

\subsection{The relationship between tractability, transitivity, and entropy}\label{secTelicSensitivity}

The following theorem asserts that if a dynamical system admits a telic problem for which any decider runs in time $2^{\Omega(n)}$, then the system has positive topological entropy.

\begin{theorem}\label{thm:entropyHardness}
Let $(X, T)$ be an efficiently discretizable topological dynamical system. If $(X, T)$ admits a telic problem such that every decider for the language runs in time $2^{\Omega(n)}$, then the topological entropy of the system is positive.
\end{theorem}

To emphasize the meta-mathematical flavor of the result, \cref{thm:entropyHardness} essentially asserts that there is a relationship between the complexity of a dynamical system and the difficulty of \emph{finding} proofs or refutations of certain basic statements about a dynamical system (the basic statements taking form ``there is no point in $g^{(n)}(I^d_n)$ whose initial orbit segment develops into $B^{(n)}$").

One novel use case of the theorem is the application of computational techniques (e.g.\ reductions) in the proof that systems have positive entropy. That is, a corollary of \cref{thm:entropyHardness} is that if $(X, T)$ admits a telic problem whose deciders run in exponential time, and $(X, T) \leq_p (Y, S)$ for any efficiently discretizable system $(Y, S)$, then $h(S)>0$.

We organize telic problems through the following notion of \emph{containment}; this allows us to speak of ``typical" telic problems coming from a dynamical system.

\begin{definition}\label{defContainment}
Fix an efficiently discretizable dynamical system $(X, T)$. We shall say a subset $A\subseteq X$ \emph{contains} a telic problem
\[\ccTELIC_{(X, T)}(v,\{h^{(i)}\},\{B^{(i)}\})\]
if $h^{(i)}:\{0,1\}^i \rightarrow A_{m(i)}\cup\{\perp\}$, and $B^{(i)} \subseteq A_{m(i)}$ for all $i \in \NN$, where $m(i) = \min\{r_i, v(i)\}$.
Let $\Tel(A)_{(X, T)}$ denote the class of all telic problems contained in $A$.\footnote{Note $A$ need not be invariant under $T$, and in fact need not be computable: we are not fixing a set \emph{and then} defining a telic problem so that $g^{(i)}:I_i^d \rightarrow A$, and $B^{(i)} \subseteq A$.}
\end{definition}

\begin{remark}
We comment that ``small sets" can contain the telic problems if $B^{(i)} = \emptyset$ and $h^{(i)}(I_i^d) = \{\perp\}$ for all $i \leq N$ for sufficiently large $N$; these are the telic problems which are ``eventually nontrivial" and can be contained in sets of small volume.
\end{remark}

A subset $A$ of state space $X$ is \emph{pure} if every telic problem contained in $A$ can be solved in polynomial time, i.e.\ $\Tel(A)_{(X, T)} \subset \ccP$. If a set $A$ is not pure then it is \emph{impure} and $\Tel(A)_{(X, T)} \not\subset \ccP$.
This naming is designed to suggest that pure sets have a special property that any telic problem contained in them is solvable in polynomial time, while impure sets must contain a mix of easy and hard telic problems.

\begin{theorem}[Dichotomy]\label{thmTopDichotomy}
Let $X \subset \RR^d$ be compact and connected, and let $(X, T)$ be an efficiently discretizable and transitive topological dynamical system with inverse.
Then either $X$ is pure, or every pure set $A \subset X$ is contained in the complement of an open and dense subset of $X$.
\end{theorem}

We prove \cref{thm:entropyHardness} and \cref{thmTopDichotomy} in the following subsections.

\subsection{Proof of \cref{thm:entropyHardness}}

The theorem can be proved via a variety of rather elementary arguments; we show the theorem by lower-bounding the size of $(n,\varepsilon)$-separated sets. Logarithms are taken to be base 2.

\begin{proof}[Proof of  \cref{thm:entropyHardness}]
Let $X\subset \RR^\ell$ be compact metric space and $T:X\rightarrow X$ continuous with the system $(X,T)$ efficiently discretizable. Suppose for a contradiction that there is a telic problem
\[
\ccTELIC_{(X, T)}\left(v,\{g^{(i)}\},\{B^{(i)}\}\right)
\]
whose solvers all halt in time $2^{\Omega(n)}$, while $h(T)\coloneqq h_{\text{top}}(T) = 0$.

For the purpose of this proof we assume $v(n)=n$, since $v(n) = Cn$ for some constant $C$, and all constants are absorbed in the asymptotic arguments used here.

Let $Y_n \coloneqq \{g^{(n)}(s): s \in I^d_n, g^{(n)}(s) \neq \perp\} \subseteq X_n$. Put $x_s = g^{(n)}(s)\in Y_n$.
Let $\rho_n(x, y) \coloneqq \max_{0\leq j < n} \rho(T^j(x), T^j(y))$ be the usual Bowen metric. For $\varepsilon>0$ let $N(n,\varepsilon)$ be the maximal size of an $(n, \varepsilon)$-separated subset of $X$. Recall the topological entropy of $(X, T)$ is defined as
\[
h(T) = \lim_{\varepsilon\rightarrow 0}\left(\limsup_{n\rightarrow \infty}\frac{1}{n} \log N(n,\varepsilon)\right) \geq 0.
\]
Let
\[
\Phi_n(x) = \ccB(\ccM_T(\langle x\rangle, 1^n, 1^n)),
\]
for $x \in X_{r_i}$, noting $\ccM_T$ returns values at the working resolution $v(n)=n$, contained in $X_n$. Hence, for $\varepsilon = 1/2^{n+1}$, there is an $n_0$ such that for all $n \geq n_0$, $\rho_n(x, y) \leq \varepsilon$ implies $\Phi_n(x) = \Phi_n(y)$.

By hypothesis, \emph{any} machine deciding the telic problem runs in time at least $2^{\alpha n}$ for some $\alpha > 0$ and sufficiently large $n$. Furthermore, the sequence of maps $\{g^{(n)}\}_{n=1}^\infty$ is feasibly queryable, so, using the neighborhood preservation property and \cref{lemEfficientSetDecidability}, two points $s, s' \in I_n^d$ are adjacent (neighboring in the lattice $I_n^d$) iff $x_s, x_{s'}$ are adjacent in $X$.
In addition, checking whether $\Phi_n(x_s) = 1$ takes time $\poly(n)$.

Using these observations, the $2^{\alpha n}$ runtime of any decider implies $2^{\alpha n}/\poly(n) \geq 2^{\alpha'n}$ checks of distinct points $s \in I^d_n$, where $\alpha'>0$. This implies that at least $2^{\alpha' n}$ adjacent points $x_s, x_{s'} \in Y_n$ satisfy $\rho_n(T(x_s), T(x_{s'}))\geq \varepsilon$. For otherwise, the fact that $\rho_n(x, y) \leq \varepsilon$ implies $\Phi_n(x) = \Phi_n(y)$, and $s, s'$ adjacent implies $x_s, x_{s'}$ adjacent (by the neighborhood preservation property), so this property can be used to avoid checking both $x_s$ and $x_{s'}$. This gives the lower bound: $2^{\alpha' n} \leq N(n, \varepsilon_n)$ for some $\alpha'>0$ and $\varepsilon_n = 1/2^{n+1}$.

Letting
\[
h(\varepsilon)=\limsup_{n\rightarrow \infty}\frac{1}{n} \log N(n,\varepsilon),
\]
and $h_{\Top}(T) = \lim_{\varepsilon\rightarrow 0}h(\varepsilon)$, notice for fixed $k \in \NN$, $\varepsilon_k = 2^{-k}$,
\[
h(\varepsilon_k)=\limsup_{n\rightarrow \infty}\frac{1}{n} \log N(n,\varepsilon_k) \geq \frac{1}{k}  N(k, \varepsilon_k).
\]
Using the lower bound $2^{\alpha' k} \leq N(k, \varepsilon_k)$,
\[
h(\varepsilon_k)\geq \frac{1}{k}\log(2^{\alpha' k}) = \alpha' \log 2.
\]
This holds for every $k$, so along the sequence $\varepsilon_k \downarrow 0$,
\[
h_{\text{top}}(T) \geq \limsup_{k\rightarrow\infty} h(\varepsilon_k)\geq \alpha' \log 2 > 0.
\]
\end{proof}

\subsection{Proof of \cref{thmTopDichotomy}}
We begin with some lemmas.

\begin{lemma}\label{lemHT->Open}
Let $(X, T)$ be an efficiently discretizable topological dynamical system with inverse. If $(X, T)$ admits a hard telic problem, it must admit a hard telic problem that cannot be contained in a set with empty interior.
\end{lemma}
\begin{proof}
Suppose $(X, T)$ admits a hard telic problem $\ccT=\ccTELIC_{(X, T)}\left(v,\{h^{(i)}\},\{B^{(i)}\}\right)$ that is contained in a set with nonempty interior. Then for all $i \in \NN$, $B^{(i)} = \emptyset$ or $B^{(i)} = b$ for some point $b \in X$ (not necessarily in $X_{v(i)}$). We first note that $B^{(i)}$ can be fattened to a small open ball $U^{(i)} \supset B^{(i)}$ such that $\ccD_{v(i)}(U^{(i)}) = \ccD_{v(i)}(B^{(i)})$. This would prove the theorem according to our definitions.

Below we give a more enlightening construction that makes use of the assumption that $T$ is efficiently discretizable with inverse.

Use $\ccT$ to define a new telic problem
\[
\ccT'=\ccTELIC_{(X, T)}(v,\{h^{(i)}\},\{A^{(i)}\}),
\]
equal to $\ccT$ for all parameters except the target sets, where we take, for every $i \in \NN$, $A^{(i)} \supset B^{(i)}$ to be the smallest open ball containing $B^{(i)}$ such that $\ccD_{v(i)}(A^{(i)})$ contains $p$ of the $2d$ points of $X_{v(i)}$ closest to $\ccD_{v(i)}(B^{(i)})$, $1 \leq p \leq 2d$, recalling $g^{(i)}:I^d_i\rightarrow X_{r_i} \cup \{\perp\}$.

We claim $\ccT'$ cannot be decided in polynomial time if $\ccT$ is hard: consider the following simple reduction from $\ccT$ to $\ccT'$. Given an instance $1^n$ of $\ccT$, solve $\ccT'$ for the same instance. If there is no solution, return $0$. Otherwise, check if the unique certificate $s \in I_n^d$ is such that $T^{n}_{v(n)}(h^{(n)}(s)) = \ccD_{v(n)}(B^{(n)})$ (the certificate is unique since $T$ is bijective). If not, redefine $\ccT$ so that $h^{(n)}(s) = \perp$, and solve the problem again. This repeats at most $p$ times, and returns the correct solution to $\ccT$.
\end{proof}

Given a class of telic problems $\Tel(A)_{(X, T)}$, it is natural to ask what the complexity of the \emph{hardest} telic problems contained in $\Tel(A)_{(X, T)}$ is.
Let $\sim$ denote the equivalence relation on the power set $\cP(X)$ of $X$, defined so that $A \sim B$ if and only if both $A$ and $B$ contain only easy telic problems, or each contain a hard telic problem.\footnote{Again, we are not concerned with the computability of the sets here.}  This induces a natural preorder on the sets, based on whether or not they contain a hard telic problem: say $A \preceq B$ if $B$ contains a hard telic problem.
Let $[A]$ label the equivalence class of $A \subseteq X$.
The preorder is upgraded to a partial order, so write $[A]<[B]$ if $B$ contains a hard telic problem but $A$ does not.

\begin{lemma}\label{lemUnion}
Let $(X, T)$ be an efficiently discretizable topological dynamical system with inverse. For every $A \subseteq X$, $[A] = [T^l(A)]$ for all $l \in \NN$.
\end{lemma}
\begin{proof}
We prove $[A] \geq [T^l(A)]$, with the case $[A] \leq [T^l(A)]$ following by identical argument since $T$ is an efficiently discretizable map with inverse.

Fix a telic problem 
\[
\ccT_1=\ccTELIC_{(X, T)}(v,\{g^{(i)}\},\{B^{(i)}\})
\]
contained in $\Tel(T^l(A))_{(X, T)}$.
Note that since $T$ is a bijection, $\ccT_1$ is in correspondence with another telic problem 
\[
\ccT_2 = \ccTELIC_{(X, T)}(v,\{T_{r_i}^{-l} \circ g^{(i)}\},\{T^{-l}(B^{(i)})\}) \in \Tel(A)_{(X, T)},
\]
noting $T$ commutes with itself, so $\{T^{-l}\circ g^{(i)}\}$ and $\{T^{-l}(B^{(i)})\}$ form valid function and target set sequences.

Consequently, any instance of $\ccT_1$ can be transformed into an instance of $\ccT_2$ by simply taking a $l$-fold composition of $T^{-1}$ over target sets $B^{(i)}$ and the functions $h^{(i)}$. And since $l$ is fixed, and $T^{-1}$ is efficiently discretizable and bijective, it follows that $\Tel(A)_{(X, T)} \subseteq \textsf{P}$ if and only if $\Tel(T^l(A))_{(X, T)} \subseteq \textsf{P}$.
\end{proof}

We now prove the dichotomy theorem \cref{thmTopDichotomy}.

\begin{proof}[Proof of \cref{thmTopDichotomy}]
Suppose $(X, T)$ admits a hard telic problem. By \cref{lemUnion} if $U \subset X$ is impure, so is $T(U)$. And $U \cup T(U)$ is impure as well.
In addition, by \cref{lemHT->Open} there is then a hard telic problem that cannot be contained in a null set, or eventually contained in a null set. Let $U$ be an open set containing such a hard telic problem. Define
\[
D \coloneqq \bigcup_{n \in \ZZ} T^n(U).
\]
Each $T^n(U)$ is open so $D$ is open, and in particular $D$ is impure. To see that $D$ is dense, let $V \subset X$ be nonempty and open. By topological transitivity there is an $n\geq 0$ such that $T^n(U) \cap V \neq \emptyset$. But $T^n(U) \subset D$, so $V \cap D \neq \emptyset$. So $D$ is an impure open and dense subset of $X$.
\end{proof}

\section{Pseudorandom dynamical systems \& further development}\label{secProblems}

We now gather conjectures and problems whose resolution would greatly impact the development of the connection between dynamics and the theory of computation.

\subsection{Pseudorandom dynamical systems}

One of the more interesting questions this paper points toward is whether there exist \emph{pseudorandom dynamical systems}. Recall from the introduction that we say an efficiently discretizable dynamical system is \emph{pseudorandom} if it admits an intractable telic problem.

The source of interest stems from the fact that if a dynamical system is pseudorandom it must possess an invariant subset with a ``stable" form of stochasticity: taking a discretization of the system does not force the system to act in a way that is distinguishable from random for all time. Instead, for short time horizons the dynamics can be used to obtain sequences of bits \emph{computationally indistinguishable} from random by using the intractable telic problem to construct a one-way function or pseudorandom generator (supposing certain average-case complexity conditions are met, cf.\ \cref{conjectureCrypto} below).
This immediately brings to mind issues of theoretical cryptography, of which we recall some basic notions before stating a conjecture.

A probabilistic polynomial time Turing machine (\ccPPT) is a Turing machine with access to a stream of random bits (i.e.\, colloquially, the Turing machine ``can flip coins"). A function $\mu(\cdot)$ is said to be \emph{negligible} if for every polynomial $p(\cdot)$ there exists some $n_0$ such that for all $n > n_0$, $\mu(n) \leq 1/p(n)$.
Let $f:\{0,1\}^*\rightarrow \{0,1\}^*$ be a polynomial-time computable function. $f$ is said to be a \emph{one-way function} if for every $\ccPPT$ algorithm $A$, there exists a negligible function $\mu$ such that for all $n \in \NN$,
\[
\PP[x \leftarrow \{0,1\}^n;y = f(x): A(1^n, y) \in f^{-1}(f(x))] \leq \mu(n).
\]
We provide one more related definition. A \emph{distributional problem} is a pair $(L, \cD)$ where $L \subseteq\{0,1\}^*$, and $\cD$ is a $\ccPPT$.
We say a distributional problem $(L, \cD)$ is \emph{almost-everywhere $\delta$ hard-on-average} if for all $\ccPPT$ $A$ and infinitely many $n \in \NN$,
\[
\PP[x \leftarrow \cD(1^n):  A(1^n, x) = L(x)] \leq 1-\delta.
\]
The Dichotomy Theorem and its corollaries essentially state that if a dynamical system admits hard telic problems, then those problems are in great abundance.
The commonality of such hard telic problems (if they exist) naturally leads to the consideration of whether this would imply the existence of one-way functions, the existence of which is necessarily closely related to the existence of hard on the average distributional problems. The following conjecture deals with this question.

Recall a topological dynamical system system $(X, T)$ is said to be \emph{topologically exact} if, for every nonempty open set $U \subset X$, there exists an $n_0 \in \mathbb{N}$ such that $T^{n_0}(U) = X$.

\begin{conjecture}\label{conjectureCrypto}
Let $(X, T)$ denote an efficiently discretizable topological dynamical system.
If $T$ is topologically exact and the system admits a hard telic problem, then there exists an almost-everywhere $\delta$ hard-on-average telic problem coming from the dynamical system.
\end{conjecture}

This conjecture acts as a precursor to showing one-way functions can be constructed from particular dynamical systems admitting hard telic problems (see also \cite{everett2024use} for related discussion).

\subsection{Additional paths for further development}

We begin with the following

\begin{problem}
Suppose it holds that whenever $(X, T)$ has positive entropy and $(X, T) \leq_p (Y, S)$, then $(Y, S)$ also has positive entropy. Under this assumption can it be shown that the Chirikov standard map has positive entropy for certain parameter values?
\end{problem}

Difficulties in studying the confluence of finite computation and dynamical systems with infinite state-spaces quickly emerge when one instead asks what conclusions can be drawn about the original system if the optimal decider for any telic problem arising from the system runs in merely \emph{super-polynomial time} rather than exponential time. Equivalently, it is hard classifying systems whose telic problems are contained in $\ccP$; it is easy to show for systems that are non-expansive for instance, but for more complicated systems (such as equicontinuous systems) the complexity remains extremely subtle. We leave some conjectures to this end.

\begin{conjecture}
If $(X, T)$ is an efficiently discretizable Morse-Smale system, then every telic problem coming from the system is contained in $\ccP$.
\end{conjecture}

\begin{conjecture}
If $(X, T)$ is an equicontinuous efficiently discretizable system, then every telic problem coming from the system is contained in $\ccP$.
\end{conjecture}

\begin{conjecture}
Let $(X, T)$ be an efficiently discretizable topological dynamical system. If $(X, T)$ admits an intractable telic problem, then there is a closed subset $E \subset X$ on which the system has sensitivity to initial conditions.
\end{conjecture}

\begin{conjecture}
Let $(X, T)$ be an efficiently discretizable topological dynamical system. If $(X, T)$ is transitive and admits an intractable telic problem, then the system is topologically mixing.
\end{conjecture}

One important problem to solve is to find the ``line of feasibility" dividing the dynamical systems that only admit easy telic problems, with those that admit hard telic problems. It is unclear if admittance of an intractable telic problem implies the system has positive entropy (\cref{thm:entropyHardness1} shows this if any decider must run in \emph{exponential time}, but for super-polynomial time the situation remains unclear).

\begin{problem}\label{probEfficientConj}
If two efficiently discretizable dynamical systems $(X, T)$ and $(Y, S)$ are conjugate, must this conjugacy be efficiently discretizable?
\end{problem}

A positive solution to Problem \ref{probEfficientConj} would in fact illuminate a deep result: conjugacies between efficiently discretizable dynamical systems must be \emph{computationally} simple. A corollary would be that if one system admits an intractable telic problem while another does not, they are not conjugate.
An answer to this question would then be tremendously valuable in determining what kind of form we can expect reductions between telic problems to take, which itself would lend itself to determining whether there exists hard telic problems.

A related problem is

\begin{problem}
If two efficiently discretizable dynamical $(X, T)$ and $(Y, S)$ are conjugate (where the conjugacy need not be efficiently discretizable), does this imply every telic problem coming from $(X, T)$ reduces (in polynomial-time) to one coming from $(Y, S)$ and vice versa?
\end{problem}

Any insight into the answers of these problems would be very interesting, revealing intimate relationships between dynamics and computation.

\bibliographystyle{abbrv}
\bibliography{references}

@article{gracca2024robust,
  title={Robust non-computability of dynamical systems and computability of robust dynamical systems},
  author={Gra{\c{c}}a, Daniel S and Zhong, Ning},
  journal={Logical Methods in Computer Science},
  volume={20},
  year={2024},
  publisher={Episciences. org}
}

@article{crisanti1994applying,
  title={Applying algorithmic complexity to define chaos in the motion of complex systems},
  author={Crisanti, Andrea and Falcioni, Massimo and Mantica, Giorgio and Vulpiani, Angelo},
  journal={Physical Review E},
  volume={50},
  number={3},
  pages={1959},
  year={1994},
  publisher={APS}
}

@article{bournez2017polynomial,
  title={Polynomial time corresponds to solutions of polynomial ordinary differential equations of polynomial length},
  author={Bournez, Olivier and Gra{\c{c}}a, Daniel S and Pouly, Amaury},
  journal={Journal of the ACM (JACM)},
  volume={64},
  number={6},
  pages={1--76},
  year={2017},
  publisher={ACM New York, NY, USA}
}

@phdthesis{galanis2014phase,
  title={Phase transitions in the complexity of counting},
  author={Galanis, Andreas},
  year={2014},
  school={Georgia Institute of Technology}
}

@article{ercsey2011optimization,
  title={Optimization hardness as transient chaos in an analog approach to constraint satisfaction},
  author={Ercsey-Ravasz, M{\'a}ria and Toroczkai, Zolt{\'a}n},
  journal={Nature Physics},
  volume={7},
  number={12},
  pages={966--970},
  year={2011},
  publisher={Nature Publishing Group UK London}
}

@article{sudoku2012chaos,
  title={The Chaos Within Sudoku},
  author={Sudoku, In},
  journal={SCIENTIFIC REPORTS},
  volume={2},
  number={725},
  pages={1},
  year={2012}
}

@article{wolpert2024stochastic,
  title={Is stochastic thermodynamics the key to understanding the energy costs of computation?},
  author={Wolpert, David H and Korbel, Jan and Lynn, Christopher W and Tasnim, Farita and Grochow, Joshua A and Karde{\c{s}}, G{\"u}lce and Aimone, James B and Balasubramanian, Vijay and De Giuli, Eric and Doty, David and others},
  journal={Proceedings of the National Academy of Sciences},
  volume={121},
  number={45},
  pages={e2321112121},
  year={2024},
  publisher={National Academy of Sciences}
}

@article{lloyd1988complexity,
  title={Complexity as thermodynamic depth},
  author={Lloyd, Seth and Pagels, Heinz},
  journal={Annals of physics},
  volume={188},
  number={1},
  pages={186--213},
  year={1988},
  publisher={Elsevier}
}

@article{crutchfield1989inferring,
  title={Inferring statistical complexity},
  author={Crutchfield, James P and Young, Karl},
  journal={Physical review letters},
  volume={63},
  number={2},
  pages={105},
  year={1989},
  publisher={APS}
}

@article{denef2007computational,
  title={Computational complexity of the landscape: Part {I}},
  author={Denef, Frederik and Douglas, Michael R},
  journal={Annals of Physics},
  volume={322},
  number={5},
  pages={1096--1142},
  year={2007},
  publisher={Elsevier}
}

@article{denef2018computational,
  title={Computational complexity of the landscape {II}—Cosmological considerations},
  author={Denef, Frederik and Douglas, Michael R and Greene, Brian and Zukowski, Claire},
  journal={Annals of Physics},
  volume={392},
  pages={93--127},
  year={2018},
  publisher={Elsevier}
}

@article{gogioso2014aspects,
  title={Aspects of statistical physics in computational complexity},
  author={Gogioso, Stefano},
  journal={arXiv preprint arXiv:1405.3558},
  year={2014}
}

@article{bennett1982thermodynamics,
  title={The thermodynamics of computation—a review},
  author={Bennett, Charles H},
  journal={International Journal of Theoretical Physics},
  volume={21},
  pages={905--940},
  year={1982},
  publisher={Springer}
}

@article{wolpert2019stochastic,
  title={The stochastic thermodynamics of computation},
  author={Wolpert, David H},
  journal={Journal of Physics A: Mathematical and Theoretical},
  volume={52},
  number={19},
  pages={193001},
  year={2019},
  publisher={IOP Publishing}
}

@misc{welsh1990computational,
  title={The computational complexity of some classical problems from statistical physics},
  author={Welsh, Dominic JA},
  journal={Disorder in physical systems},
  volume={307},
  pages={307--321},
  year={1990},
  publisher={Clarendon Press, Oxford}
}

@article{kreinovich2006towards,
  title={Towards applying computational complexity to foundations of physics},
  author={Kreinovich, Vladik and Finkelstein, Andrei M},
  journal={Journal of Mathematical Sciences},
  volume={134},
  pages={2358--2382},
  year={2006},
  publisher={Springer}
}

@article{manzano2024thermodynamics,
  title={Thermodynamics of computations with absolute irreversibility, unidirectional transitions, and stochastic computation times},
  author={Manzano, Gonzalo and Karde{\c{s}}, G{\"u}lce and Rold{\'a}n, {\'E}dgar and Wolpert, David H},
  journal={Physical Review X},
  volume={14},
  number={2},
  pages={021026},
  year={2024},
  publisher={APS}
}

@article{kardecs2022inclusive,
  title={Inclusive thermodynamics of computational machines},
  author={Karde{\c{s}}, G{\"u}lce and Wolpert, David},
  journal={arXiv preprint arXiv:2206.01165},
  year={2022}
}

@article{ouldridge2023thermodynamics,
  title={Thermodynamics of deterministic finite automata operating locally and periodically},
  author={Ouldridge, Thomas E and Wolpert, David H},
  journal={New Journal of Physics},
  volume={25},
  number={12},
  pages={123013},
  year={2023},
  publisher={IOP Publishing}
}

@article{PhysRevResearch.2.033312,
  title = {Thermodynamic costs of {Turing} machines},
  author = {Kolchinsky, Artemy and Wolpert, David H.},
  journal = {Phys. Rev. Res.},
  volume = {2},
  issue = {3},
  pages = {033312},
  numpages = {22},
  year = {2020},
  month = {Aug},
  publisher = {American Physical Society},
  doi = {10.1103/PhysRevResearch.2.033312},
  url = {https://link.aps.org/doi/10.1103/PhysRevResearch.2.033312}
}

@article{galatolo2003complexity,
  title={Complexity, initial condition sensitivity, dimension and weak chaos in dynamical systems},
  author={Galatolo, Stefano},
  journal={Nonlinearity},
  volume={16},
  number={4},
  pages={1219},
  year={2003},
  publisher={IOP Publishing}
}

@article{yampolsky2021towards,
  title={Towards understanding the theoretical challenges of numerical modeling of dynamical systems},
  author={Yampolsky, Michael},
  journal={New Zealand Journal of Mathematics},
  volume={52},
  pages={453--467},
  year={2021}
}

@article{gracca2018computing,
  title={Computing geometric {Lorenz} attractors with arbitrary precision},
  author={Gra{\c{c}}a, Daniel and Rojas, Cristobal and Zhong, Ning},
  journal={Transactions of the American Mathematical Society},
  volume={370},
  number={4},
  pages={2955--2970},
  year={2018}
}

@article{kuurka1997topological,
  title={On topological dynamics of {Turing} machines},
  author={K\r{u}rka, Petr},
  journal={Theoretical Computer Science},
  volume={174},
  number={1-2},
  pages={203--216},
  year={1997},
  publisher={Elsevier}
}

@inproceedings{kuurka2001topological,
  title={Topological dynamics of cellular automata},
  author={K\r{u}rka, Petr},
  booktitle={Codes, Systems, and Graphical Models},
  pages={447--485},
  year={2001},
  organization={Springer}
}

@article{kuurka1997languages,
  title={Languages, equicontinuity and attractors in cellular automata},
  author={K\r{u}rka, Petr},
  journal={Ergodic theory and dynamical systems},
  volume={17},
  number={2},
  pages={417--433},
  year={1997},
  publisher={Cambridge University Press}
}

@article{dudko2016poly,
  title={Poly-time computability of the {Feigenbaum Julia} set},
  author={Dudko, Artem and Yampolsky, Michael},
  journal={Ergodic Theory and Dynamical Systems},
  volume={36},
  number={8},
  pages={2441--2462},
  year={2016},
  publisher={Cambridge University Press}
}

@article{braverman2009constructing,
  title={Constructing locally connected non-computable Julia sets},
  author={Braverman, Mark and Yampolsky, Michael},
  journal={Communications in Mathematical Physics},
  volume={291},
  number={2},
  pages={513--532},
  year={2009},
  publisher={Springer}
}

@article{braverman2006non,
  title={Non-computable Julia sets},
  author={Braverman, Mark and Yampolsky, Michael},
  journal={Journal of the American Mathematical Society},
  pages={551--578},
  year={2006},
  publisher={JSTOR}
}

@book{braverman2009computability,
  title={Computability of Julia sets},
  author={Braverman, Mark and Yampolsky, Michael},
  volume={23},
  year={2009},
  publisher={Springer}
}

@article{cardona2024hydrodynamic,
  title={Hydrodynamic and symbolic models of computation with advice},
  author={Cardona, Robert},
  journal={Revista Matem{\'a}tica Iberoamericana},
  year={2024}
}

@article{bruera2024topological,
  title={Topological entropy of {Turing} complete dynamics},
  author={Bruera, Renzo and Cardona, Robert and Miranda, Eva and Peralta-Salas, Daniel},
  journal={arXiv preprint arXiv:2404.07288},
  year={2024}
}

@article{cardona2024towards,
  title={Towards a fluid computer},
  author={Cardona, Robert and Miranda, Eva and Peralta-Salas, Daniel},
  journal={Foundations of Computational Mathematics},
  pages={1--17},
  year={2025},
  publisher={Springer}
}

@article{tao2016finite,
  title={Finite time blowup for an averaged three-dimensional {Navier-Stokes} equation},
  author={Tao, Terence},
  journal={Journal of the American Mathematical Society},
  volume={29},
  number={3},
  pages={601--674},
  year={2016}
}

@article{tao2017universality,
  title={On the universality of potential well dynamics},
  author={Tao, Terence},
  journal={Dyamics of partial differential equations},
  volume={14},
  pages={219–-238},
  year={2017}
}

@article{hiura2019microscopic,
  title={Microscopic reversibility and macroscopic irreversibility: From the viewpoint of algorithmic randomness},
  author={Hiura, Ken and Sasa, Shin-ichi},
  journal={Journal of Statistical Physics},
  volume={177},
  pages={727--751},
  year={2019},
  publisher={Springer}
}

@article{dufort1997dynamics,
  title={Dynamics, complexity and computation},
  author={Dufort, Paul A and Lumsden, C},
  journal={Physical Theory in Biology. Singapore: World Scientific},
  pages={69--106},
  year={1997}
}

@article{da1991undecidability,
  title={Undecidability and incompleteness in classical mechanics},
  author={da Costa, Newton CA and Doria, Francisco A},
  journal={International Journal of Theoretical Physics},
  volume={30},
  pages={1041--1073},
  year={1991},
  publisher={Springer}
}

@article{da1994undecidable,
  title={Undecidable Hopf bifurcation with undecidable fixed point},
  author={da Costa, Newton CA and Doria, Francisco A},
  journal={International Journal of Theoretical Physics},
  volume={33},
  pages={1885--1903},
  year={1994},
  publisher={Springer}
}

@article{koiran1999closed,
  title={Closed-form analytic maps in one and two dimensions can simulate universal {Turing} machines},
  author={Koiran, Pascal and Moore, Cristopher},
  journal={Theoretical Computer Science},
  volume={210},
  number={1},
  pages={217--223},
  year={1999},
  publisher={Elsevier}
}

@article{moore1991generalized,
  title={Generalized shifts: unpredictability and undecidability in dynamical systems},
  author={Moore, Cristopher},
  journal={Nonlinearity},
  volume={4},
  number={2},
  pages={199},
  year={1991},
  publisher={IOP Publishing}
}

@article{moore1990unpredictability,
  title={Unpredictability and undecidability in dynamical systems},
  author={Moore, Cristopher},
  journal={Physical Review Letters},
  volume={64},
  number={20},
  pages={2354},
  year={1990},
  publisher={APS}
}

@article{zurek1989algorithmic,
  title={Algorithmic randomness and physical entropy},
  author={Zurek, Wojciech H},
  journal={Physical Review A},
  volume={40},
  number={8},
  pages={4731},
  year={1989},
  publisher={APS}
}

@article{cardona2021constructing,
  title={Constructing {Turing} complete {Euler} flows in dimension 3},
  author={Cardona, Robert and Miranda, Eva and Peralta-Salas, Daniel and Presas, Francisco},
  journal={Proceedings of the National Academy of Sciences},
  volume={118},
  number={19},
  pages={e2026818118},
  year={2021},
  publisher={National Acad Sciences}
}

@article{cardona2022turing,
  title={Turing universality of the incompressible {Euler} equations and a conjecture of Moore},
  author={Cardona, Robert and Miranda, Eva and Peralta-Salas, Daniel},
  journal={International Mathematics Research Notices},
  volume={2022},
  number={22},
  pages={18092--18109},
  year={2022},
  publisher={Oxford University Press}
}

@article{braverman2015space,
  title={Space-bounded {Church-Turing} thesis and computational tractability of closed systems},
  author={Braverman, Mark and Schneider, Jonathan and Rojas, Crist{\'o}bal},
  journal={Physical review letters},
  volume={115},
  number={9},
  pages={098701},
  year={2015},
  publisher={APS}
}

@article{PhysRevLett.83.1463,
  title = {Computational Complexity for Continuous Time Dynamics},
  author = {Siegelmann, Hava T. and Ben-Hur, Asa and Fishman, Shmuel},
  journal = {Phys. Rev. Lett.},
  volume = {83},
  issue = {7},
  pages = {1463--1466},
  numpages = {0},
  year = {1999},
  month = {Aug},
  publisher = {American Physical Society},
  doi = {10.1103/PhysRevLett.83.1463},
  url = {https://link.aps.org/doi/10.1103/PhysRevLett.83.1463}
}

@article{bournez2013computation,
  title={Computation with perturbed dynamical systems},
  author={Bournez, Olivier and Gra{\c{c}}a, Daniel S and Hainry, Emmanuel},
  journal={Journal of Computer and System Sciences},
  volume={79},
  number={5},
  pages={714--724},
  year={2013},
  publisher={Elsevier}
}

@article{boyer2015mu,
  title={$\mu$-limit sets of cellular automata from a computational complexity perspective},
  author={Boyer, Laurent and Delacourt, Martin and Poupet, Victor and Sablik, Mathieu and Theyssier, Guillaume},
  journal={Journal of Computer and System Sciences},
  volume={81},
  number={8},
  pages={1623--1647},
  year={2015},
  publisher={Elsevier}
}

@article{de2018characterization,
  title={Characterization of sets of limit measures of a cellular automaton iterated on a random configuration},
  author={De Menibus, Benjamin Hellouin and Sablik, Mathieu},
  journal={Ergodic Theory and Dynamical Systems},
  volume={38},
  number={2},
  pages={601--650},
  year={2018},
  publisher={Cambridge University Press}
}

@article{rojas2025algorithmic,
  title={On the algorithmic descriptive complexity of attractors in topological dynamics},
  author={Rojas, Cristobal and Sablik, Mathieu},
  journal={Ergodic Theory and Dynamical Systems},
  pages={1--28},
  year={2025},
  publisher={Cambridge University Press}
}

@article{barrett2007predecessor,
  title={Predecessor existence problems for finite discrete dynamical systems},
  author={Barrett, Chris and Hunt III, Harry B and Marathe, Madhav V and Ravi, SS and Rosenkrantz, Daniel J and Stearns, Richard E and Thakur, Mayur},
  journal={Theoretical Computer Science},
  volume={386},
  number={1-2},
  pages={3--37},
  year={2007},
  publisher={Elsevier}
}

@article{barrett2006complexity,
  title={Complexity of reachability problems for finite discrete dynamical systems},
  author={Barrett, Christopher L and Hunt III, Harry B and Marathe, Madhav V and Ravi, SS and Rosenkrantz, Daniel J and Stearns, Richard E},
  journal={Journal of Computer and System Sciences},
  volume={72},
  number={8},
  pages={1317--1345},
  year={2006},
  publisher={Elsevier}
}

@inproceedings{bournez2013computability,
  title={Computability and computational complexity of the evolution of nonlinear dynamical systems},
  author={Bournez, Olivier and Gra{\c{c}}a, Daniel S and Pouly, Amaury and Zhong, Ning},
  booktitle={The Nature of Computation. Logic, Algorithms, Applications: 9th Conference on Computability in Europe, CiE 2013, Milan, Italy, July 1-5, 2013. Proceedings 9},
  pages={12--21},
  year={2013},
  organization={Springer}
}

@article{cecconi2003complexity,
  title={Complexity characterization of dynamical systems through predictability},
  author={Cecconi, Fabio and Falcioni, Massimo and Vulpiani, Angelo},
  journal={arXiv preprint nlin/0307013},
  year={2003}
}

@inproceedings{grassberger2005complexity,
  title={Complexity and forecasting in dynamical systems},
  author={Grassberger, Peter},
  booktitle={Measures of Complexity: Proceedings of the Conference, Held in Rome September 30--October 2, 1987},
  pages={1--21},
  year={2005},
  organization={Springer}
}

@inproceedings{cotler2024computational,
  title={Computational Dynamical Systems},
  author={Cotler, Jordan and Rezchikov, Semon},
  booktitle={2024 IEEE 65th Annual Symposium on Foundations of Computer Science (FOCS)},
  pages={166--202},
  year={2024},
  organization={IEEE}
}

@article{stoop2004complexity,
  title={Complexity of dynamics as variability of predictability},
  author={Stoop, Ruedi and Stoop, Norbert and Bunimovich, Leonid},
  journal={Journal of statistical physics},
  volume={114},
  pages={1127--1137},
  year={2004},
  publisher={Springer}
}

@article{rojas2019computational,
  title={Computational intractability of attractors in the real quadratic family},
  author={Rojas, Cristobal and Yampolsky, Michael},
  journal={Advances in Mathematics},
  volume={349},
  pages={941--958},
  year={2019},
  publisher={Elsevier}
}

@article{binder2006computational,
  title={On computational complexity of Siegel Julia sets},
  author={Binder, Ilia and Braverman, Mark and Yampolsky, Michael},
  journal={Communications in mathematical physics},
  volume={264},
  number={2},
  pages={317--334},
  year={2006},
  publisher={Springer}
}

@inproceedings{papadimitriou2016computational,
  title={On the computational complexity of limit cycles in dynamical systems},
  author={Papadimitriou, Christos H and Vishnoi, Nisheeth K},
  booktitle={Proceedings of the 2016 ACM Conference on Innovations in Theoretical Computer Science},
  pages={403--403},
  year={2016}
}

@article{rosenkrantz2024synchronous,
  title={Synchronous dynamical systems on directed acyclic graphs: Complexity and algorithms},
  author={Rosenkrantz, Daniel J and Marathe, Madhav V and Ravi, SS and Stearns, Richard E},
  journal={ACM Transactions on Computation Theory},
  volume={16},
  number={2},
  pages={1--34},
  year={2024},
  publisher={ACM New York, NY}
}

@article{ogihara2017computational,
  title={Computational complexity studies of synchronous Boolean finite dynamical systems on directed graphs},
  author={Ogihara, Mitsunori and Uchizawa, Kei},
  journal={Information and Computation},
  volume={256},
  pages={226--236},
  year={2017},
  publisher={Elsevier}
}

@article{gora1988computers,
  title={Why computers like Lebesgue measure},
  author={G{\'o}ra, Pawel and Boyarsky, Abraham},
  journal={Computers \& Mathematics with Applications},
  volume={16},
  number={4},
  pages={321--329},
  year={1988},
  publisher={Elsevier}
}

@article{bowen1971entropy,
  title={Entropy for group endomorphisms and homogeneous spaces},
  author={Bowen, Rufus},
  journal={Transactions of the American Mathematical Society},
  volume={153},
  pages={401--414},
  year={1971}
}

@inproceedings{dinaburg1970correlation,
  title={A correlation between topological entropy and metric entropy},
  author={Dinaburg, Efim Isaakovich},
  booktitle={Doklady Akademii Nauk},
  volume={190},
  pages={19--22},
  year={1970},
  organization={Russian Academy of Sciences}
}

@book{weinberger2020computers,
  title={Computers, Rigidity, and Moduli: The Large-Scale Fractal Geometry of Riemannian Moduli Space},
  author={Weinberger, Shmuel},
  year={2020},
  publisher={Princeton University Press}
}

@article{freedman2009complexity,
  title={Complexity classes as mathematical axioms},
  author={Freedman, Michael H},
  journal={Annals of mathematics},
  pages={995--1002},
  year={2009},
  publisher={JSTOR}
}

@article{freedman1998topological,
  title={Topological views on computational complexity},
  author={Freedman, M},
  journal={Documenta Mathematica-Extra Volume ICM},
  pages={453--464},
  year={1998}
}

@article{nabutovsky1996algorithmic,
  title={Algorithmic unsolvability of the triviality problem for multidimensional knots},
  author={Nabutovsky, Alexander and Weinberger, Shmuel},
  journal={Commentarii Mathematici Helvetici},
  volume={71},
  pages={426--434},
  year={1996},
  publisher={Springer}
}

@article{braverman2006computing,
  title={Computing over the reals: Foundations for scientific computing},
  author={Braverman, Mark and Cook, Stephen},
  journal={Notices of the AMS},
  volume={53},
  number={3},
  pages={318--329},
  year={2006}
}

@book{blank1997discreteness,
  title={Discreteness and continuity in problems of chaotic dynamics},
  author={Blank, Michael L},
  volume={161},
  year={1997},
  publisher={American Mathematical Soc.}
}

@article{eckmann1985ergodic,
  title={Ergodic theory of chaos and strange attractors},
  author={Eckmann, J-P and Ruelle, David},
  journal={Reviews of modern physics},
  volume={57},
  number={3},
  pages={617},
  year={1985},
  publisher={APS}
}

@article{everett2,
  title={Correspondences in computational and dynamical complexity {II}: forcing complex reductions},
  author={Everett, Samuel},
  year={2026}
}

@article{everett3,
  title={Correspondences in computational and dynamical complexity {III}: feasibility and topological entropy},
  author={Everett, Samuel},
  year={2026}
}

@article{everett2024use,
  title={On the use of dynamical systems in cryptography},
  author={Everett, Samuel},
  journal={Chaos, Solitons \& Fractals},
  volume={183},
  pages={114952},
  year={2024},
  publisher={Elsevier}
}

@book{robinson1998dynamical,
  title={Dynamical systems: stability, symbolic dynamics, and chaos},
  author={Robinson, Clark},
  year={1998},
  publisher={CRC press}
}

@book{walters2000introduction,
  title={An introduction to ergodic theory},
  author={Walters, Peter},
  volume={79},
  year={2000},
  publisher={Springer Science \& Business Media}
}

@book{arora2009computational,
  title={Computational complexity: a modern approach},
  author={Arora, Sanjeev and Barak, Boaz},
  year={2009},
  publisher={Cambridge University Press}
}

@book{papadimitriou,
  title={Computational complexity},
  author={Papadimitriou, Christos},
  year={1994},
  publisher={Addison-Wesley}
}

@book{percus2006computational,
  title={Computational complexity and statistical physics},
  author={Percus, Allon and Istrate, Gabriel and Moore, Cristopher},
  year={2006},
  publisher={OUP USA}
}

@book{ko2012complexity,
  title={Complexity theory of real functions},
  author={Ko, Ker},
  year={2012},
  publisher={Springer Science \& Business Media}
}

@book{weihrauch2012computable,
  title={Computable analysis: an introduction},
  author={Weihrauch, Klaus},
  year={2012},
  publisher={Springer Science \& Business Media}
}

@book{li2008introduction,
  title={An introduction to Kolmogorov complexity and its applications},
  author={Li, Ming and Vit{\'a}nyi, Paul and others},
  volume={3},
  year={2008},
  publisher={Springer}
}

@book{brin2002introduction,
  title={Introduction to dynamical systems},
  author={Brin, Michael and Stuck, Garrett},
  year={2002},
  publisher={Cambridge university press}
}

@inproceedings{buss1990predictability,
  title={On the predictability of coupled automata: An allegory about Chaos},
  author={Buss, Sam and Papadimitriou, Christos H and Tsitsiklis, NJ},
  booktitle={Proceedings [1990] 31st Annual Symposium on Foundations of Computer Science},
  pages={788--793},
  year={1990},
  organization={IEEE}
}

@inproceedings{applebaum2010cryptography,
  title={Cryptography by cellular automata or how fast can complexity emerge in nature?},
  author={Applebaum, Benny and Ishai, Yuval and Kushilevitz, Eyal},
  booktitle={ICS},
  pages={1--19},
  year={2010}
}

@article{sutner1995computational,
  title={On the computational complexity of finite cellular automata},
  author={Sutner, Klaus},
  journal={Journal of Computer and System Sciences},
  volume={50},
  number={1},
  pages={87--97},
  year={1995},
  publisher={Elsevier}
}

@article{sutner2012computational,
  title={Computational classification of cellular automata},
  author={Sutner, Klaus},
  journal={International Journal of General Systems},
  volume={41},
  number={6},
  pages={595--607},
  year={2012},
  publisher={Taylor \& Francis}
}

@InProceedings{SutnerComplexityOfCA,
author="Sutner, Klaus",
editor="Csirik, J.
and Demetrovics, J.
and G{\'e}cseg, F.",
title="The computational complexity of cellular automata",
booktitle="Fundamentals of Computation Theory",
year="1989",
publisher="Springer Berlin Heidelberg",
address="Berlin, Heidelberg",
pages="451--459"
}

@inproceedings{hartmanis1983sparse,
  title={Sparse sets in {NP-P}: {EXPTIME} versus {NEXPTIME}},
  author={Hartmanis, Juris and Sewelson, Vivian and Immerman, Neil},
  booktitle={Proceedings of the fifteenth annual ACM symposium on Theory of computing},
  pages={382--391},
  year={1983}
}

@book{basu2006algorithms,
  title={Algorithms in real algebraic geometry},
  author={Basu, Saugata and Pollack, Richard and Roy, Marie-Fran{\c{c}}oise},
  year={2006},
  publisher={Springer}
}

\end{document}